\documentclass[11pt,a4paper]{amsart}
\usepackage[svgnames,dvipsnames]{xcolor}
\usepackage{amssymb,amsmath,amsthm,enumerate}
\usepackage{hyperref}
\usepackage[left=2.5cm,right=2.5cm,top=2.5cm,bottom=3cm]{geometry}
\usepackage{caption}
\usepackage[scr=rsfso]{mathalfa}
\usepackage{upgreek}
\usepackage{nicefrac}
\usepackage{graphicx}
\usepackage{float}
\usepackage{tikz}

\newcommand\myshade{90}
\colorlet{mylinkcolor}{violet}
\colorlet{mycitecolor}{YellowOrange}
\colorlet{myurlcolor}{Aquamarine}

\hypersetup{
  linkcolor  = mylinkcolor!\myshade!black,
  citecolor  = mycitecolor!\myshade!black,
  urlcolor   = myurlcolor!\myshade!black,
  colorlinks = true,
}

\usepackage{dsfont}

\usepackage{scalerel,stackengine}
\stackMath
\newcommand\reallywidehat[1]{%
	\savestack{\tmpbox}{\stretchto{%
			\scaleto{%
				\scalerel*[\widthof{\ensuremath{#1}}]{\kern-.6pt\bigwedge\kern-.6pt}%
				{\rule[-\textheight/2]{1ex}{\textheight}}
			}{\textheight}%
		}{0.5ex}}%
	\stackon[1pt]{#1}{\tmpbox}%
}

\newcommand{\bR}{\mathbb{R}^3}
\newcommand{\bRp}{\mathbb{R}_+}
\newcommand{\bRr}{\mathbb{R}}
\newcommand{\bRfp}{\mathbb{R}^3 \times \bRp}
\newcommand{\bRfr}{\mathbb{R}^3 \times \bRr}
\newcommand{\bS}{\mathbb{S}^2}
\newcommand{\bSp}{\mathbb{S}^2_+}
\newcommand{\mS}{\mathcal{S}}
\newcommand{\md}{\mathrm{d}}
\newcommand{\dpo}{d_\alpha(r,R)}
\newcommand{\dub}{d_\alpha(r,R) }

\newcommand{\g}{\gamma}

\newcommand{\la}{\langle}
\newcommand{\ra}{\rangle}

\newcommand{\Ls}{L^1}

 \newcommand{\D}{\mathcal{D}}
  \newcommand{\Dt}{\tilde{\mathcal{D}}}

\newcommand{\Fou}{\mathcal{F}}

\newcommand{\E}{\left(\frac{E}{m}\right)}

\newcommand{\cb}{\mathcal{C}}
\newcommand{\mA}{\mathcal{A}}
\newcommand{\mAg}{\mathcal{A}}

\newcommand{\K}{K_1}

\newcommand{\A}{A}

\newcommand{\hag}{h_{\alpha,\g}}
\newcommand{\kk}{\beta}
\newcommand{\test}{I^\delta}
\newcommand{\testp}{I'^\delta}

\newcommand{\param}{ \sigma, r, R}
\newcommand{\params}{ - \sigma, 1 - r, R}
\newcommand{\paramp}{  \sigma', r', R'}

\newcommand{\domparam}{\bSp \times [0,1]^2}
\newcommand{\domparamgen}{\bS \times [0,1]^2}

\newcommand{\iall}{\iiint\limits_{\substack{(v,I)\in \bRfp \\ (v_*,I_*)\in \bRfp \\ (\param) \in \domparam}}}
\newcommand{\iallintro}{\int_{(\bRfp)^2}\int_{\domparam}}

\newcommand{\istarintro}{\int_{\bRfp}\int_{\domparam }}

\newcommand{\istargenintro}{\int_{ \bRfp} \int_{ \domparamgen }}

\newcommand{\iRt}{\iiint\limits_{\substack{(v,I)\in \bRfp \\ (v_*,I_*)\in \bRfp \\ (r,\tilde{R},\sigma) \in [0,1]\times\bRfp\times\bSp }}}

\newcommand{\iI}{\iint\limits_{\substack{(I,I_*)\in \bRp\!\times\bRp \\ r\in[0,1]}}}
\newcommand{\ien}{\int_{\substack{(I,I_*)\in (\bRp)^2 }}}
\newcommand{\iv}{\iint\limits_{\substack{(v,v_*)\in \bR\times\bR \\ (\tilde{R},\sigma) \in  [0,\infty)\times\bSp}}}
\newcommand{\iRts}{\int\limits_{  (\tilde{R},\sigma) \in  \bRfp\times\bSp}}

\newcommand{\dall}{ \md \sigma \, \md R \, \md r \, \md v_* \md I_* \, \md v \, \md I}
\newcommand{\dstar}{ \md \sigma \, \md R \, \md r \, \md v_* \md I_*}

\newcommand{\dstartilde}{ \md \tilde{R} \, \md(\cos \tilde{\theta}) \md \omega \, \md r \, \md v'_* \md I'_*}

\newcommand{\dRt}{ \md \sigma \, \md \tilde{R} \, \md r \, \md v_* \md I_* \, \md v \, \md I}
\newcommand{\dI}{   \md r \,   \md I_*  \, \md I}
\newcommand{\dv}{ \md \sigma \, \md \tilde{R} \,   \md v_*  \md v}
\newcommand{\dvvs}{  \md v_* \md I_* \, \md v \, \md I}

\newcommand{\dRts}{ \md \sigma \, \md \tilde{R}  }

\newcommand{\V}{\frac{v+v_*}{2}}

\newcommand{\Et}{\tilde{E}}

\newcommand{\R}{\mathcal{R}}
\newcommand{\RI}{\mathcal{R}}
\newcommand{\Prim}{\mathcal{P}}

\DeclareMathOperator*{\esssup}{ess \, sup}

\newcommand{\cI}{\tilde{c}_1}
\newcommand{\cIIs}{\tilde{c}_2}

 \newcommand{\fR}{f_R}
  \newcommand{\fS}{f_S}

  \newcommand{\Csvel}{C_{\text{vel}}^{\text{sm}}}
  \newcommand{\Csen}{C_{\text{en}}^{\text{sm}}}
  
  \newcommand{\Crgvel}{C_{\text{vel}}^{\text{reg}}}
  \newcommand{\Crgen}{C_{\text{en}}^{\text{reg}}}
  
  \newcommand{\CveltS}{C_{\text{vel};t_0}^{S}}
  \newcommand{\CenttS}{C_{\text{en};t_0}^{S}}
  
  \newcommand{\CvelzS}{C_{\text{vel};0}^{S}}
  \newcommand{\CenzS}{C_{\text{en};0}^{S}}
  
  \newcommand{\Clvel}{\mathcal{C}_{\text{vel}}}
  \newcommand{\Clen}{\mathcal{C}_{\text{en}}}
  
    \newcommand{\Cpr}{C^{\text{pr}}}
    
      \newcommand{\Ccmvel}{C_{\text{vel}}^{\text{c}}}
 
\allowdisplaybreaks

\newtheorem{theorem}{Theorem}

\newtheorem{lemma}[theorem]{Lemma}
\newtheorem{proposition}[theorem]{Proposition}

\newtheorem{remark}{Remark}

\begin{document}

\bibliographystyle{abbrv}

	\title[Regularity theory]{Regularity Theory for the Space Homogeneous Polyatomic Boltzmann Flow}
	
		\author[R. Alonso]{Ricardo Alonso}
	\address{Texas A\&M, Division of arts \& sciences, Education City, Doha, Qatar }
	\email{ricardo.alonso@qatar.tamu.edu}

	\author[M. \v Coli\'{c}]{Milana \v Coli\'{c}}
		\address{Institute of Mathematics and Scientific Computing, 
		University of Graz, 
		Heinrichstr. 36, 
		8010 Graz, Austria}
	\address{Department of Mathematics and Informatics, 
		Faculty of Sciences, University of Novi Sad, 
		Trg Dositeja Obradovi\'ca 4, 21000 Novi Sad, Serbia}
	\email{milana.colic@uni-graz.at}

\begin{abstract}
	In this paper, we study the polyatomic Boltzmann equation based on continuous internal energy, focusing on physically relevant collision kernels of the hard potentials type with integrable angular part. We establish three main results: smoothing effects of the gain collision operator, propagation of velocity and internal energy first-order derivatives of solutions, and exponential decay estimates for singularities of the initial data. These results ultimately lead to a decomposition theorem, showing that any solution splits into a smooth part and a rapidly decaying rough component.  
	\end{abstract}

\maketitle

\medskip


	
	
\section{Introduction}

The Boltzmann equation is known as  a robust framework for describing non-equilibrium processes in gas flows \cite{torrilhon2016modeling,Rugg-book-2,Kusto-book,Str,Gio}. In applications, it plays a central role  particularly in scenarios where classical fluid dynamics models, like the Navier-Stokes-Fourier laws, become inadequate. Examples include rarefied regime (vacuum technology, high altitude flights) and microscopic setting (Micro- and Nano-Electro-Mechanical-Systems).

Mathematical theory of the Boltzmann equation has been an active research field, resulting in a broad understanding of the analytical and numerical properties of its solution \cite{V}. However, most results are restricted to the physical setting of a single monatomic gas described by a single equation with a simple mechanism of particle interactions possessing symmetries that facilitate analysis. For application purposes, there is a clear need for studying more intrigant gases.

This paper develops a regularity theory for the space-homogeneous Boltzmann equation describing a single polyatomic gas, when the gas particles are composed of several atoms. Moments and higher integrability properties, which provide quantitative bounds on the norms of the solutions in Lebesgue spaces, have already been obtained \cite{MPC-IG-poly, MC-Alonso-Gamba, MC-Alonso-Lp, MC-Alonso-Pesaro}. Building upon these results, regularity emerges as a natural continuation, as was already known for monatomic gases in \cite{L,BD,MV}  that uses previously established $L^p$ bounds with $p=2$. We now pursue this direction in the context of polyatomic gases. The resulting theory provides a stepping stone toward deeper understanding of the equation, particularly regarding convergence towards equilibrium, and forms the basis for the analysis of stability and convergence of numerical methods \cite{AGT-S}.

We now briefly introduce the Boltzmann equation for polyatomic gases, with the relevant functional spaces. Along with the presentation of our main results in Section \ref{Sec: main}, we review the known results in the context of regularity theory for the  classical case of a single monatomic gas in the space homogeneous setting.

\subsection{Continuous approach to modelling a polyatomic gas and the Boltzmann equ\-ation} A polyatomic molecule is defined as a molecule composed of   two or more atoms. Its internal  structure causes rotational and vibrational degrees of freedom, besides the translational ones \cite{Rugg-book-2}. Continuous approach makes a modelling choice to capture all the energies, apart from the kinetic one, involved in the various motions of a polyatomic molecule by one continuous internal  energy variable $I \in \mathbb{R}_+$ \cite{LD-Bourgat,LD-Toulouse,DesMonSalv}. 

In the space homogeneous setting, distribution function $f: = f(t,v,I)\geq 0$   depends on time $t\geq 0$ and molecular velocity-internal energy pair $(v,I) \in \bR\times \bRp$. Its evolution is governed by the Boltzmann equation 
\begin{equation}\label{BE}
	\begin{split}
		\partial_t f(t,v,I) &= Q(f(t,\cdot,\cdot), f(t,\cdot,\cdot))(v,I)
	\end{split}
\end{equation}
where   the  collision operator  $Q(f, f)(v,I)$  encodes the effect of particles' collisions on the (tempo\-ral) change of the unknown $f$.  It  is a non-local integral operator that acts only on the  microscopic variables of $f$, in the present case $(v,I)$. 

The functional properties of the collision operator,  and, consequently, the solution to the corresponding Boltzmann equation, are strongly influenced by the nature of particle interactions or the assumptions regarding the collision kernel.  Grounded on  the applications, it is common to consider a factorized form of the collision kernel, where the angular part is separated from the kinetic part. Taking this idea, the present paper focuses on the cut-off setting requiring the angular part to be integrable, and implying that the collision operator $Q(f, f)(v,I)$  can be  represented  as the \emph{gain} minus the \emph{loss}, with the latter  being  local in the unknown $f$ and proportional to the collision frequency $	\nu[f]$,
\begin{equation}\label{BE split}
	\begin{split}
		Q(f,f)(v,I) = Q^+(f, f)(v,I) - f(v,I) \	\nu[f](v,I).
	\end{split}
\end{equation}
The kinetic part is assumed to be  a positive power $  \g \in (0,2]$    of the colliding molecule pair total  (kinetic plus internal) energy, an assumption known as hard potentials. Is is worthwhile to emphasize that the collision operator in the polyatomic setting depends on the structure of the polyatomic molecule or its degrees of freedom, manifested through a parameter $\alpha>-1$ (to  appear in the definition of the collision operator \eqref{coll operator}).  

Parameters of the Boltzmann equation \eqref{BE}--\eqref{BE split} can be linked to experimental data through moment methods used to model the transport properties of the gas. Evaluations of the collision operator \eqref{BE split}, described in \cite{MPC-Dj-S,MPC-Dj-T,MPC-Dj-T-O,MPC-SS-non-poly}, show that the polyatomic parameter $\alpha>-1$ corresponds to the polytropic specific heat of the gas, while the parameter $\g\in (0,2]$ is related to the shear viscosity exponent. Moreover, incorporating frozen collisions   \cite{MC-Alonso-frozen} allows matching the Prandtl number and, under certain conditions, the bulk viscosity. These properties validate the physical consistency and practical relevance of the polyatomic Boltzmann model \eqref{BE}--\eqref{BE split}.

Besides being successfully studied in the space homogeneous setting, results that we will review in Section~\ref{Sec: Prelim}, the mathematical analysis of the Boltzmann equation with continuous internal energy \eqref{BE} has recently attracted increasing attention. For instance, in the perturbation framework for the space dependent problem, compactness properties of the associated linearized Boltzmann operator have been studied in \cite{Bern, Brull, Bern-nous}. The global well-posedness of bounded mild solutions near global equilibrium on the torus has been established in \cite{Duan-Li}, while \cite{Ko-Son} constructs large-amplitude solutions with small initial relative entropy and proves their convergence toward the global equilibrium state at an exponential rate.

We now briefly introduce the functional spaces used in the subsequent analysis.

\subsection{Functional spaces}
We start by defining the Lebesgue brackets as in \cite{MPC-IG-poly},
\begin{equation}\label{brackets}
	\la v, I \ra = \sqrt{1 + \frac{1}{2}  |v|^2  + \frac{1}{ m}  I  }, \quad \text{for} \ \ m >0, \ v\in\bR, \ \text{and} \ \ I \in \mathbb{R}_+ = [0,\infty).
\end{equation}
Weighted  Lebesgue spaces  $L^p_{ k}(\bR\times\bRp)$,    with the weight  \eqref{brackets} of the order $k\geq 0$,  are   defined by the norm, for  $1\leq p<\infty$,
 \begin{equation}\label{Lp i}
 	 \| f \|_{L^p_{ k}(\bR\times\bRp)}
 	  =   \| f \la \cdot \ra^{k} \|_{L^p(\bR\times\bRp)} =   \left(  \int_{\bRfp} \left(  |f(v,I)| \la v, I \ra^{k} \right)^p \md v\, \md I  \right)^{1/p},
 \end{equation}
 while for $p=\infty$ the following definition is used
 \begin{equation}\label{L inf i}
  \| f \|_{L^\infty_{ k}(\bR\times\bRp)}  = \esssup\limits_{(v,I) \in \bRfp} \left(  |f(v,I)| \la v, I \ra^{k} \right).
 \end{equation}
When $k=0$, we omit  writing an index, i.e. $L^p_0\equiv L^p$. 

Finally, the first-order homogeneous Sobolev space in the velocity variable $	\dot{H}^{1}_{v}(\bRfp)$ is defined by the norm 
\begin{equation}\label{H1v}
 \| f \|_{\dot{H}^{1}_{ v}}^2   =  	\sum_{i=1}^{3}   \| f \|_{\dot{H}^{1}_{v_i}}^2,  \quad \text{where} \      \| f \|_{\dot{H}^{1}_{v_i}}  := \| \partial_{v_i} f \|_{{L}^{2}}, \ i=1,2,3.
\end{equation}
Let us stress that for a time dependent function $f(t,v,I)$, the norms \eqref{Lp i}, \eqref{L inf i} and \eqref{H1v} are taken at each particular time $t$, for example $\|f\|_{L^p} = \| f(t,\cdot)\|_{L^p}$.

When refereeing to the integrability of angular collision kernel,  we mean 
\begin{equation}\label{b integrablity}
	\| b \|_{L^1} = 2 \pi \int_{0}^1 b(x) \, \md x < \infty, \quad \text{and} \quad 	\| b \|_{L^2}^2 =  { 2 \pi }  \int_{0}^1 b(x)^2 \, \md x < \infty
\end{equation}

  Sometimes the weight will be expressed in terms of the following notation:
\begin{equation}\label{nmb plus}
	\eta^+ := \eta + \varepsilon, \quad \text{for any } \varepsilon>0 \ \text{arbitrary small}.
\end{equation}

\subsection{Overview of the previous study on  well-posedness, moments  and higher inte\-gra\-bility    }\label{Sec: Prelim}

In this section, we recall specific results from \cite{MC-Alonso-Gamba,MC-Alonso-Lp,MC-Alonso-Pesaro} concerning the Boltzmann equation \eqref{BE}--\eqref{BE split}, with the collision kernel given by the cut-off regime with the  angular part $b\in L^1$ and  hard-potentials in the total molecular energy  of the power $\gamma\in(0,2]$ (see \eqref{coll kernel assumpt} for an explicit expression), with  a parameter $\alpha>-1$, and with initial data
\begin{equation}\label{BE in}
	f(0, v, I)=f_0(v,I).
\end{equation}

\subsubsection{Existence and uniqueness theory}

Existence and uni\-que\-ness of global in time  classical solutions to the space homogeneous Boltzmann equation \eqref{BE} is established in the following theorem.
\begin{theorem}[Existence and Uniqueness,  Theorem 7.2 for $M=0$, $P=1$ in \cite{MC-Alonso-Gamba}, or Theorem 3.8 in \cite{MC-Alonso-Pesaro}]\label{Th well poss}
	Consider the collision kernel in the cut-off and hard-potentials form, see \eqref{coll kernel assumpt}, and initial data \eqref{BE in} satisfying 
	\begin{equation}\label{Omega}
		f_0(v,I) \in	\tilde{\Omega} = \left\{ g(v,I) \in L^1_2: \ g \geq 0, \ 0< \| g \|_{L^1_{0}}  < \infty, \
		\| g \|_{L^1_{2^+}}   < \infty \right\} \subset L^1_2.
	\end{equation}
	Then the Cauchy problem \eqref{BE}-\eqref{BE in} has a unique solution in $\mathcal{C}([0,\infty), \tilde{\Omega})  \cap \mathcal{C}^1((0,\infty), L^1_2)$.
\end{theorem}
This  result is obtained using   an approach of \cite{Alonso-cooling} based on  ODE theory on Banach spaces \cite{Martin-ODE}. The method consists in:  {(i)} proving that the collision operator satisfies the following three conditions: H\"older continuity,  Sub-tangent and One-sided Lipschitz conditions, which allows to find unique solution having  initially $2+2\g$ moments \cite{MPC-IG-poly}, and  {(ii)} relaxing the assumption on initial data to $2^+$, using the fact that  the collision operator is one-sided Lipschitz assuming only $2^+$  moments \cite{MC-Alonso-Gamba}, which allows to construct an approximating sequence of solutions drawn from the Step  {(i)} and pass to the limit to find solutions in the bigger space, as Theorem \ref{Th well poss} shows.

\subsubsection{$L^1$-theory} The $L^1$-theory  regards ge\-ne\-ration and propagation of moments associated to the solution of the Boltzmann equation. The ge\-ne\-ration or creation of moments  means that the solution of the Boltzmann equation initially  having the finite and strictly positive $L^1_0$-moment (mass) and finite  $L^1_2$-moment   (energy), instantaneously creates, for any time $t>0$), $L^1_k$-moment of any order $k>2$. Propagation of moments of the solution to the Boltzmann equation means that, if initially $L^1_k$-moment is finite, then it will remain finite uniformly in time.

More precisely, let $C^{\text{gen}}$ and $\Cpr$ denote generic constants explicitly computed in \cite{MC-Alonso-Gamba}, depending on $k$, $\gamma$, $\alpha$, the angular part $b$, and the conservative moments: mass  $\| f \|_{L^1_0} = \| f_0 \|_{L^1_0}$ and energy $\| f \|_{L^1_2} = \| f_0 \|_{L^1_2}$. The following theorem holds.
\begin{theorem}[$L^1$-theory, Theorem 6.2 for $M=0$, $P=1$ in \cite{MC-Alonso-Gamba}, or Theorem 3.6 in \cite{MC-Alonso-Pesaro}]\label{Th L1} If $f(t,v,I)$ is a solution of the Boltzmann equation \eqref{BE}  with \eqref{BE in}, then for  any $t>0$ and $k>2$,
		\begin{enumerate}
		\item (Polynomial moments generation estimate.) 
\begin{equation}\label{poly gen}
	\| f \|_{L^1_k} \leq C^{\text{gen}} \left(1  +  t^{-\frac{k-2}{\g}} \right). 
\end{equation}
		\item (Polynomial moments propagation estimate.) If moreover $	\| f_0 \|_{L^1_k} < \infty$,  then
		\begin{equation}\label{poly prop}
			\| f \|_{L^1_k} \leq \max \left\{  e \,	\| f_0 \|_{L^1_k}, \Cpr \right\}.
		\end{equation}
where	$e$ is the Euler's number.
	\end{enumerate}
\end{theorem}
These results rely on a dominant effect of the negative contribution of the collision operator (loss term) with respect to the positive one (gain term), in terms of moments. They are reached by averaging the gain term over the set of variables $(\param)$  known as Povzner-type estimate for the classical Boltzmann equation, and by using    moment-interpolation techniques.

\subsubsection{$L^p$-theory}

Higher-integrability theory is developed in \cite{MC-Alonso-Lp} for a mixture of solely polyatomic gases, and is complemented in \cite{MC-Alonso-Pesaro} with the study of $L^p$-tails for a single polyatomic gas. The idea is to extend  results of Theorem  \ref{Th L1} from $L^1$ to $L^p$ for $p> 1$, including the case $p=\infty$. 

An interesting feature of the polyatomic Boltzmann model is that the actual range of $p>1$ depends on the polyatomic parameter $\alpha>-1$ and on the chosen collision kernel. In particular, for the collision kernel chosen in this paper (below in \eqref{coll kernel assumpt}) with hard potentials of exponent $\gamma \in (0,2]$,  the range of admissible $p$ can be explicitly computed from the condition (5.2) in \cite{MC-Alonso-Lp},
\begin{align}
	& \text{(1)} \ \text{when} \ \  p\in(1,\infty), \ \text{then} \ p \, \alpha>-1 \ \text{and} \label{range p} \\
	& \qquad   \text{(a)} \ \text{if} \  \alpha<\gamma/2, \text{then} \ p < \frac{1+\g/2}{\g/2-\alpha},  \quad \text{or} \ 
	    \text{(b)}  \ \text{if}  \ \alpha<-1/2,  \text{then} \  p < - \frac{1 }{1+2 \alpha}; \nonumber\\
		& \text{(2)} \ \text{when} \ \  p = \infty, \ \text{then} \  \alpha > \g/2.  \label{range p inf}
\end{align}
Since the relevant case considered in this paper is $p=2$, the above conditions simplify to the following condition on $\alpha$,
\begin{equation*}
	\alpha > \frac{\gamma}{4} - \frac{1}{2}, \quad \gamma \in (0,2].
\end{equation*}

Relying on $L^1_k$-moment propagation result (Theorem \ref{Th L1}, item (2)), the $L^p$-propagation result consist in proving that a solution of the Boltzmann equation with initially bounded $L^p_k$-moment and $L^1_{k+\g+1}$-moment will have bounded $L^p_k$-moment  at any time $t>0$. Then, using  $L^p_0$-propagation result and both generation and propagation in $L^1$ (Theorem \ref{Th L1}), the $L^p$-tails ge\-ne\-ration property holds in the following sense.  If a solution  to the Boltzmann equation  initially has $L^p_0$- and $L^1_{\g+1}$-moment bounded, then $L^p_k$-moment will emerge after some time $t_0 > 0$. The emergence will happen instantaneously, i.e. for any time $t>0$, if we additionally assume that   $L^1_{k+\g+1}$-moment is  bounded initially. More precisely, the following statement holds.
\begin{theorem}[Theorem 9.1 for $P=1$ and the collision kernel \eqref{coll kernel assumpt} in \cite{MC-Alonso-Lp}, and Theorem 3.17 in \cite{MC-Alonso-Pesaro}]\label{Th prop Lp} 
 	Let  $k\geq0$, $0<\g\leq2$, $C$ denote a generic constant,  and   $f(t,v,I)$ be a solution of the Boltzmann  equation   	\eqref{BE}--\eqref{coll kernel assumpt} in the sense of Theorem \ref{Th well poss}, with $p \in (1,\infty]$ from \eqref{range p}--\eqref{range p inf}.
 	\begin{enumerate}
 			\item ($L^p_k$-propagation estimate.) If $\| f_0    \|_{L^p_{k}}< \infty$ and $\| f_0    \|_{L^1_{\max\{2, k+ \g +1\}}}< \infty$, then
	\begin{equation}\label{Lp propagation}
	\| f  \|_{L^p_{k}} \leq \max\left\{  \| f_0    \|_{L^p_{k}}, \Cpr_p \right\}, \quad \text{for} \ \ t\geq 0,
\end{equation}
 		\item ($L^p$-tails   generation estimate.)  If $\| f_0    \|_{L^p}< \infty$ and $\| f_0    \|_{L^1_{\max\{2, \g +1\}}}< \infty$, then
 		\begin{enumerate}
 			\item for $k>\max\{0,1-\g\}$ and $t_0>0$ it holds
 		\begin{equation}\label{Lp generation t0}
 			\| f \|_{L^p_k} \leq C^{\text{gen}}_{p; t_0} \left(1  +  (t-t_0)^{-\frac{k}{\g}} \right), \qquad  t>t_0,
 		\end{equation}
 	\item if additionally $\| f_0    \|_{L^1_{\max\{2, k+ \g +1\}}}< \infty$, $k\geq0$, then
 	 		\begin{equation}\label{Lp generation t}
 		\| f \|_{L^p_k} \leq C^{\text{gen}}_{p; 0}  \left(1  +  t^{-\frac{k}{\g}} \right), \qquad  t>0.
 	\end{equation}
 		\end{enumerate}
 	\end{enumerate}
\end{theorem}
The proof exploits the differential inequality approach \cite{MV}, whose strategy is to split the collision operator weak form into parts corresponding to   gain and   loss terms, and then to find appropriate estimates.  The estimate on the gain part relies on finding a suitable representation through the averaging operator, while the the estimate on the loss term uses a priori conservation laws of mass and energy and entropy. Since all the involved constants are explicit,  the case $p = \infty$  is found as a limit of the case $p < \infty$.

\section{Presentation of the main results}\label{Sec: main}

The goal of this paper is to provide a detailed analysis of propagation of smoothness and exponential decay of singularities for solutions to the polyatomic Boltzmann equation \eqref{BE}--\eqref{BE split}. Our approach relies on the regularizing effects of the gain term in the Boltzmann collision operator. We restrict the attention to the first-order derivatives in the velocity and internal energy variables.

Within the monatomic framework, when molecular velocity is the only microscopic variable, the smoothing properties of the gain operator have been extensively studied \cite{L, W, BD, LU, MV,LD-Parma}. Initial study of these   effects required collision kernels to be smooth with support avoiding certain points \cite{L}, with the proof relying on the stationary phase method. An alternative strategy was  later provided in \cite{W}  employing the properties of the generalized Radon transform. These works  are in the $L^1$-$L^2$ setting and they were used in the theory of $L^p$-propagation \cite{MV}. 

A second type of estimates is the one from  \cite{BD}, which is in a $L^2$-$L^2$ framework and is used for the propagation of regularity \cite{MV}. This approach requires $L^2$-integrability of the angular part,  a condition that was relaxed in \cite{AGT} to only $L^1$-integrability by means of a commutator formula. In general, these methods are based   on  the Fourier transform for the gain part with Maxwell molecules introduced in \cite{Bob-Exact-solutions-Fourier}.

Our first main result addresses the smoothing effects of the gain collision operator $Q^+(f,g)(v,I)$ defined in \eqref{BE split}, with respect to the variables of the polyatomic Boltzmann model, namely the velocity variable $v$ and the internal energy variable $I$.

\begin{theorem}[Main result I: Smoothing effects of the gain operator]\label{inc-reg}
	Let $\gamma\in[0,2]$. \\
	\noindent 1. (Smoothing with respect to the velocity variable.) Let $\alpha>-1$, $b\in L^{2}$ in the sense of \eqref{b integrablity}, and  $f,\,g \in L^{1}_{\gamma}\cap L^{2}_{(\gamma+2)^+}(\bRfp)$. Then,
	\begin{equation}\label{eq Th smooth vel}
		\| Q^+(f,g) \|_{\dot{H}^{1}_{v}(\bRfp)}\leq \Csvel \, \|b\|_{L^{2}}\| f \|_{L^{2}_{(\gamma+2)^+}}\| g \|_{L^{2}_{(\gamma+2)^+}}.
	\end{equation}
	\noindent 2. (Smoothing with respect to the internal energy variable.) 	
  Let  $\alpha>0$,  $b \in L^1$ in the sense of \eqref{b integrablity}. Then, for any  $\delta\geq\max\{ \left(\frac{1}{2} - \alpha\right)^+,0\}$, and $f, g \in L^2_{(\g+2\delta+1/2)^+}(\bRfp)$, the following estimate holds 
		\begin{equation}\label{eq Th smooth int en}
	\|  \test \, \partial_{I}		Q^+(f,g)(v,I)  \|_{L^2(\bRfp)}
			\leq \Csen \,  \| b \|_{\Ls} \,  \| f\|_{L^2_{(\g+2\delta+1/2)^+}}    \| g\|_{L^2_{(\g+2\delta+1/2)^+}}.
		\end{equation}
	The constants $\Csvel$, $\Csen>0$ are  estimated at the end of the proof, in \eqref{C smooth vel} and \eqref{C smooth I}.
\end{theorem}
Part (a) is proven in Section \ref{Sec: proof smooth v}, while Part (b) is proven in Section \ref{Sec: proof smooth I}.  The smoothing effect with respect to the velocity variable, Part (a),  is inspired by \cite{BD} that uses Fourier transform techniques, under the assumption $b\in L^2$. The smoothing effects with respect to the  internal energy variable in Part (b), is restricted by the regularity of the kernel of the gain operator. Namely, the strong form of the gain operator (see \eqref{coll gain operator}) contains a multiplicative term $I^\alpha$. To make the singularity of  its first derivative at the origin $L^2$-integrable, the multiplicative factor $I^\delta$ in the estimate \eqref{eq Th smooth int en} is added.

 The smoothing effects of the gain operator fall within the study of functional properties of the collision operator applied to a generic function, not necessarily related to a solution of the Boltzmann equation. Examining these functional properties has been recognized as a key element in the analysis within the space-homogeneous framework. A direct implication is that, following the so-called  differential inequality  approach, these estimates, when applied to the equation, yield understanding of the behavior of its solutions \cite{MV}. In a broader context, the properties of the collision operator are important for various applications, including the analysis of the spatially inhomogeneous Boltzmann equation, the Vlasov-Boltzmann equation, and the Boltzmann equation coupled with fluid dynamics or other kinetic equations.
 
Our second main result follows the aforementioned  differential inequality  approach to study regularity properties of the solution to the polyatomic Boltzmann equation. 

\begin{theorem}[Main result II: Propagation of regularity for solutions]\label{Th regularity}
Let $\g \in [0,2]$, $\alpha>0$ and $b\in L^1$  in the sense of \eqref{b integrablity}.  Let $f \in L^\infty([0,\infty), L^2_{(\gamma+2)^+}(\bRfp) )$ be a solution of the Boltzmann equation \eqref{BE} with  the collision kernel \eqref{coll kernel assumpt}. \\
\noindent 1. (Regularity with respect to the velocity variable.) Consider an initial data
$ f_0  \in \dot{H}^{1}_{v}$. Then, for any time $t\geq0$, the following estimate holds
\begin{equation}\label{regularity v}
	\| f \|_{\dot{H}^{1}_{v}}^2  \leq  \max \left\{ \| f_0 \|_{\dot{H}^{1}_{v}}^2, \Crgvel \right\}.  
\end{equation}
\noindent 2. (Regularity with respect to the internal energy variable.) 	 For  $\delta=\max\{ \left(\frac{1}{2} - \alpha\right)^+,0\}$, consider an initial data such that
$ I^{\delta} \partial_{I} f_0\in {{L}^{2}}$. Then, for any time $t\geq0$, the following estimate holds
	\begin{equation}\label{regularity I}
	\| I^{\delta} \partial_{I} f \|_{{L}^{2}}^2 \leq     \max \left\{    \| I^{\delta} \partial_{I} f_0 \|_{{L}^{2}}^2, \Crgen \right\}.
\end{equation}
The constants  $\Crgvel, \Crgen>0$ are estimated at the end of the proof, in \eqref{C reg vel} and \eqref{C reg en}.
\end{theorem}
The proof of this theorem is given in Section \ref{Sec: proof reg}. Inspired by the idea of \cite{AGT}, the smoothing effects of the gain operator   are complemented by the commutator formula in the velocity variable, in order to relax the assumption on $b$ to only $L^1$-integrability. Thus, we include an intermediate Proposition \ref{inc-reg-b1} containing the velocity commutator formula for the gain operator, when $\alpha>0$, $\g \in [0,2]$, $b\in L^1$. 

To develop the regularity result for the solution of the Boltzmann equation, the angular kernel is then decomposed  into two parts: $L^2$ part for which the gain of regularity of Theorem \ref{inc-reg}, Part (a), can be applied, and  $L^1$ part for which there is no gain of regularity, but can be made as small as desired. This latter piece will be controlled by the negative part coming from the loss operator. Thus,  the suitable commutator formulas  for the loss operator in both velocity and internal energy variables is developed   in Proposition \ref{inc-loss-reg-b1}, for $\alpha>-1$, $\g \in [0,2]$, $b\in L^1$.

After the propagation of smoothness, the aim is to obtain the estimates on the decay of singularities of initial datum, leading to the so-called decomposition theorem that states that a solution can be decomposed into a sum of the smooth part and a non-smooth (rough) part whose amplitude decays exponentially fast, proven in \cite{MV,AGT} in the monatomic framework. This concept is extended to a single polyatomic gas in our third main result. 


\begin{theorem}[Main result III: Decomposition theorem]\label{Th decomp} Let $\alpha>0$, $\delta=\max\{ \left(\frac{1}{2} - \alpha\right)^+,0\}$,  $\g \in [0,2]$, and $b\in L^2$  in the sense of \eqref{b integrablity}. The solution of the Boltzmann equation can be decomposed into two non-negative parts:
	\begin{equation}\label{f decomposed}
		\begin{split}
			f(t,v,I)  = \fR(t,v,I) + \fS(t,v,I), \quad \text{for} \ \  t\geq  0.
		\end{split}
	\end{equation}
The rough part $\fR(t,v,I)$ satisfies
\begin{equation}\label{rough est}
	\fR(t,v,I)  \leq e^{-\mA (t-t_0)} f_{t_0}(v,I), 
\end{equation}
where $\mA >0$ is the constant from \eqref{coll freq lower est}.
The smooth part $	\fS(t,v,I)$ satisfies the following: \\
(a) If $f_0 \in L^1_2 \cap L^2$, then
\begin{equation}\label{fS gen}
	\| 	 	\fS \|_{ \dot{H}^1_v } \leq \CveltS, \quad \text{and} \quad		\| 	I^\delta \partial_{I}	\fS \|_{L^2} \leq \CenttS, \quad t \geq t_0>0.
\end{equation} 
(b) If  $f_0 \in L^1_{(2\g+3)^+} \cap L^2_{(\g+2)^+}$, then  
\begin{equation}\label{fS prop}
	\| 	 	\fS \|_{ \dot{H}^1_v } \leq \CvelzS, \quad \text{and} \quad		\| 	I^\delta \partial_{I}	\fS \|_{L^2} \leq \CenzS, \quad t\geq  0.
\end{equation}
In all statements, the  generic constants are  estimated at the end of the proof.
\end{theorem}
The proof is given in Section \ref{Sec proof decomp}. We emphasize that this theorem
shows the weak diffusion process of the collisional effects happening in the cut-off models.  Namely, singularities are damped in an infinite time, as shown by \eqref{rough est}. This is in contrast to the non-cutoff models, at least for monatomic gases \cite{non-cut}, in which the singularities disappear instantaneously.  

 \section{Space homogeneous Boltzmann equation in the continuous  setting describing polyatomic gases }

 In this work, binary collisions are the main interaction principle and their description as a building block for defining collision operator is given in the upcoming Subsection \ref{Sec: bin coll}.
 
 \subsection{Binary collisions }\label{Sec: bin coll}
 In this paper, collisions between the two molecules is assumed to conserve  momentum and the total energy. More precisely, 
 if the colliding particles of mass $m>0$ have pre-collisional velocity-internal energy pairs $(v', I')$, $(v'_*, I'_*)$ that change to $(v, I)$ and $(v_*, I_*)$, respectively, then the following conservation laws are assumed to hold
 \begin{equation}\label{coll CL}
 	v' + v'_*  = v + v_*, \qquad \frac{m}{2} |v'|^2 + I' +  \frac{m}{2} |v'_*|^2 + I'_* = \frac{m}{2} |v|^2 + I +  \frac{m}{2} |v_*|^2 + I_*.
 \end{equation}
 The energy law \eqref{coll CL}$_2$ can be rewritten in terms of the relative velocity $u:=v-v_*$,
 \begin{equation}\label{en}
 	E' = E := \frac{m}{4} |u|^2 + I +  I_*.  
 \end{equation}  
 A possible way to express the primed quantities in \eqref{coll CL} in terms of non-primed is to use the so-called Borgnakke-Larsen procedure  \cite{Bor-Larsen, LD-Bourgat}  that introduces a set of parameters, namely the \emph{angular} variable and \emph{energy exchange} variables 
 \begin{equation}\label{Bor-Lar param}
 	\sigma\in\bS, \quad r, R \in [0,1].
 \end{equation}
 Then, 
 \begin{equation}\label{coll rules}
 	\begin{alignedat}{2}
 		v'  &= \V + \sqrt{\frac{ R  \, E}{m}} \sigma,  \qquad
 		I' &&=r (1-R)E, \\
 		v'_{*} &= \V -   \sqrt{\frac{  R \,  E}{m}} \sigma,
 		\qquad I'_* &&= (1-r)(1-R) E. 
 	\end{alignedat}
 \end{equation}

 \subsection{Collision operator}	 Collision transformation \eqref{coll rules} allows to define the collision operator for distribution functions $f:= f(t,v,I)\geq 0$ and $g:= g(t,v,I)\geq0$, and a parameter $\alpha >-1$, as
 \begin{multline}\label{coll operator general}
 	Q(f,g)(v,I) =\istargenintro \left\{ f(v',I')g(v'_*,I'_*) \left(\frac{I I_*}{I' I'_*}\right)^{\alpha}   - f(v,I)g(v_*, I_*) \right\} \\ \times  \mathcal{B}(v,v_*,I,I_*,\param)  \, \dpo  \, \dstar,
 \end{multline}
 where the collision kernel is $ \mathcal{B}(v,v_*,I,I_*,\param) \geq 0$ a.e. satisfies microreversibility assumptions 
 \begin{equation}\label{coll kernel general}
 	\mathcal{B}(v,v_*,I,I_*,\param)  =  \mathcal{B}(v_*,v,I_*,I, \params) =  \mathcal{B}(v',v'_*,I',I'_*,\paramp), 
 \end{equation}
 with the primed quantities defined in   \eqref{coll rules} and  additionally
 \begin{equation}\label{primed param}
 	\sigma'=\frac{u}{\left|u\right|} =: \hat{u} \in \bS, \qquad	r'=\frac{I}{I+I_*} = \frac{I}{E-\frac{m}{4}\left| u \right|^2} \in [0,1], \qquad R'=\frac{m \left|u\right|^2}{4 E} \in [0,1],
 \end{equation} 
 and where 	the function $\dpo$,  closely related to the weight factor $I^\alpha$,   is given by 
 \begin{equation}\label{fun r R}
 	\dpo =	  r^{\alpha}(1-r)^{\alpha} \, (1-R)^{2\alpha+1} \, \sqrt{R}.
 \end{equation} 
 In this paper, we will specify the collision kernel and thus will consider more specific form of the collision operator than \eqref{coll operator general}.
 
 \subsection{Assumption on the collision kernel} On physical grounds \cite{MPC-Dj-T, MPC-Dj-T-O},
 the collision kernel \eqref{coll kernel general}  is assumed to take a factorized form of hard potentials in the collisional total energy \eqref{en}  with an  integrable angular part, i.e.
 \begin{equation}\label{coll kernel assumpt}
 	\mathcal{B}(v,v_*,I,I_*,\param)  =  b(\hat{u}\cdot\sigma) \E^{\g/2}, \quad \g  \in (0,2], 
 \end{equation}
 and  $b(\hat{u}\cdot\sigma) \geq 0$ 
 assumed integrable  on the unit sphere  $\bS$ and, without   loss of generality, supported on $\bSp$,
 \begin{equation*}
 	\bSp = \left\{ \sigma \in \bS: \hat{u} \cdot \sigma \geq 0 \right\},  \  \text{for   fixed} \ \hat{u}:= \frac{u}{|u|},
 \end{equation*}
 that is,  by means of spherical coordinates, 
 \begin{equation}\label{b integrable}
 	\| b \|_{L^1} = \int_{\bSp} b(\hat{u}\cdot\sigma) \, \md \sigma  = 2 \pi \int_{0}^1 b(x) \, \md x < \infty.
 \end{equation}
 \begin{remark} Note that the restriction $\hat{u} \cdot  \sigma \geq 0$ in the definition of the collision operator \eqref{coll operator} holds without the loss of generality for the factorized form of the collision kernel \eqref{coll kernel assumpt} since the present framework is the single-species and thus colliding particles are indistinguishable. More precisely, since the product $f' f'_*$ appearing in the quadratic Boltzmann collision operator is invariant under the change of variables $(\sigma, r) \rightarrow (- \sigma, 1-r)$ and the reminding  terms are symmetric with respect to it,  $b(\hat{u} \cdot  \sigma)$ can be replaced by its symmetrized version $(b(\hat{u} \cdot  \sigma) + b( - \hat{u} \cdot  \sigma)) \mathds{1}_{\hat{u} \cdot  \sigma \geq 0}$, as it was possible for the classical monatomic case. 
 \end{remark}
 The integrability assumption \eqref{b integrable} implies a well-defined constant:
 \begin{equation}\label{kappa}
 	\kappa = 	\| b \|_{L^1}  \| d_\alpha \|_{L^1([0,1]^2)}. 
 \end{equation}
 Thus, assumption \eqref{coll kernel assumpt} allows to rewrite the collision operator \eqref{coll operator general}, 	for $ \alpha>-1$,
 \begin{multline}\label{coll operator}
 	Q(f,g)(v,I) =\istarintro \left\{ f(v',I')g(v'_*,I'_*) \left(\frac{I I_*}{I' I'_*}\right)^{\alpha}   - f(v,I)g(v_*, I_*) \right\} \\ \times b(\hat{u}\cdot\sigma) \, \E^{\g/2} \, \dpo  \, \dstar.
 \end{multline}
 The collision operator  \eqref{coll operator} can be understood as a difference between the gain part $	Q^+(f,g)(v,I)$ and the loss part $	Q^-(f,g)(v,I)$, where the gain operator is given by the integral form
 \begin{multline}\label{coll gain operator}
 	Q^+(f,g)(v,I) =\istarintro  f(v',I')g(v'_*,I'_*) \left(\frac{I I_*}{I' I'_* }\right)^{\alpha}    \\ \times b(\hat{u}\cdot\sigma) \, \E^{\g/2} \, \dpo  \, \dstar
 \end{multline}
 while the loss operator is local in $f$, more precisely,
 \begin{equation}\label{coll loss operator}
 	Q^-(f,g)(v,I) =  f(v,I) \	\nu[g](v,I),
 \end{equation}
 where the collision frequency $\nu[g](v,I)$ is given by
 \begin{equation}\label{coll fr}
 	\nu[g](v,I) = \istarintro g(v_*, I_*) \, b(\hat{u}\cdot\sigma) \, \E^{\g/2} \, \dpo  \, \dstar.
 \end{equation}
 
 \subsubsection*{The weak form of the collision operator}
 For any suitable test function $\chi(v,I)$, the pre-post transformation \eqref{coll rules}--\eqref{primed param} allows to make the primed quantities passing to the test function, i.e.
 \begin{multline}\label{weak form pre-post}
 	\int_{\bRfp} Q(f,g)(v, I) \, \chi(v,I) \, \md v\, \md I  
 	\\	=  \istarintro  \chi(v', I')   \, f(v, I) \, g(v_*, I_*) \,
 	b(\hat{u}\cdot\sigma) \, \E^{\g/2}  \dpo \, \dall,
 \end{multline}
 which enables the following weak form of the collision operator \eqref{coll operator} 
 \begin{multline}\label{weak form}
 	\int_{\bRfp} Q(f,g)(v, I) \, \chi(v,I) \, \md v\, \md I  
 	=  \iallintro \left\{  \chi(v', I')  - \chi(v, I) \right\} \, f(v, I) \, g(v_*, I_*) 
 	\\ \times b(\hat{u}\cdot\sigma) \, \E^{\g/2} \, \dpo \, \dall.
 \end{multline}
 Details can be found in \cite{DesMonSalv, MC-Alonso-Gamba}.

\section{Commutator formulas}

\begin{lemma}[Lower bound on the collision frequency] \label{Prop: coll freq} Assume $g \in L^1_{2} \cap L^2$ and $\g \in(0,2]$. Then, there exists an explicit  constant ${c}^{lb}[g]  > 0$   such that the following lower bound on the collision frequency  $	\nu[g](v,I) $ given in \eqref{coll fr}  holds
	\begin{equation}\label{coll freq lower est}
		\text{if} \ \g >0 \quad \text{then} \quad \nu[g](v,I) \,   \geq \mAg \, \la v, I \ra^{\g}, \ \text{with} \  \mAg =\kappa \, c^{lb}[g], 
	\end{equation} 
	and if $\g=0$, then  $ \nu[g](v,I) =  \kappa  \| g \|_{L^1}$, 
	where $\kappa$ is given in \eqref{kappa}.   
\end{lemma}
This lemma is a version of Lemma 3.12 from \cite{MC-Alonso-Pesaro} or Lower Bound Lemma 7.1 from \cite{MC-Alonso-Lp}. The condition on the entropy is changed to the  more restrictive $L^2$ condition, which is   natural   in the present setting.  For a more detailed discussion, the interested reader is  referred  to    Lemma 3.11 from \cite{MC-Alonso-Pesaro}.

\begin{proposition}[Commutator for the loss operator]\label{inc-loss-reg-b1}
	For $\gamma\in[0,2]$, $\alpha> -1$, $b\in L^{1}(\bSp)$, let $f$ and $g$ be suitable functions  that make all expressions well defined.  Then the derivative with respect to the velocity variable $v$ and   internal energy variable $I$ of the loss part \eqref{coll loss operator}  is written as a sum of the primary term  containing the derivative of the input function $f$ and the reminder,
	\begin{equation}\label{der Q- wrt vi}
		\begin{split}
		\partial_{v_i} 	Q^-(f,g) = \nu[g] \, \partial_{v_i} f +f  \, \partial_{v_i} \nu[g], 
	\quad \text{and} \quad
			\partial_{I} 	Q^-(f,g) = \nu[g] \, \partial_{I} f +f  \, \partial_{I} \nu[g],
		\end{split}
	\end{equation}
where the primary term is estimated from below by means of Lemma \ref{Prop: coll freq}, and the reminder term is estimated  from above by the following expressions, when $\g>0$,  with notation 
\eqref{kappa} and \eqref{constant Ca},
	\begin{align}
& |\partial_{v_i} \nu[g](v,I)  | \leq \Clvel \, 	\| g \|_{L^{2}_{(\g+3/2)^+}}  \la v, I \ra^{(\g+3/2)^+},  \quad \Clvel :=  \frac{\g \, \kappa}{2}  \, \cb_{(\g -1)/2}, \label{est derivative} \\
 & |\partial_{I} \nu[g](v,I)  | \leq \Clen \,	\| g \|_{L^{2}_{(\g+1/2)^+}}  \la v, I \ra^{(\g+1/2)^+}, \quad \Clen :=  \frac{\g \, \kappa}{2}  \, \cb_{\g/2 -1}, \label{est derivative energy}
	\end{align}
and  $\partial_{v_i} \nu[g] = \partial_{I}\nu[g] =0$ when $\g=0$.
\end{proposition}
\begin{proof} Taking derivative of the collision frequency $\nu$ given in  \eqref{coll fr} with respect to $v_i$,
\begin{equation}
\partial_{v_i} \left(	\nu[g](v,I) \right) = \frac{\g \, \kappa}{4}    \int_{\bRfp}   g(v_*, I_*) \,   \E^{\g/2-1}  (v_i - v_{*i})  \, \md v_* \, \md I_*,
\end{equation}
and estimating  $\frac{1}{2}|v_i-v_{*i}|\leq \frac{1}{2}|v-v_{*}|\leq (E/m)^{1/2}$, implies
\begin{equation}
|	\partial_{v_i} \left(	\nu[g](v,I) \right) | \leq \frac{\g \, \kappa }{2}   \int_{\bRfp}   g(v_*, I_*) \,   \E^{(\g-1)/2}    \, \md v_* \, \md I_*.
\end{equation}
The Cauchy-Schwartz inequality implemented as in \eqref{constant Ca estimate}  implies  the statement \eqref{est derivative}.

Similarly for the derivative   with respect to $I$,
\begin{equation}
	\partial_{I} \left(	\nu[g](v,I) \right) = \frac{\g \, \kappa}{2}     \int_{\bRfp}   g(v_*, I_*) \,   \E^{\g/2-1}    \, \md v_* \, \md I_*,
\end{equation}
\eqref{constant Ca estimate}  implies  the statement \eqref{est derivative}.
\end{proof}

 We build the following commutator to treat the regularity property of the solution to the Boltzmann equation as stated in Theorem \ref{Th regularity}, for the critical case $b \in L^1$. If $b\in L^2$, then the following proposition can be omitted and a direct application of Theorem \ref{inc-reg} will suffice to prove Theorem \ref{Th regularity}. 

\begin{proposition}[Velocity commutator for the gain operator]\label{inc-reg-b1}
	For $\gamma\in[0,2]$, $\alpha > 0$, $b\in L^{1}(\bSp)$, let $f$ and $g$ be suitable functions  that make all expressions well defined.  Then the derivative with respect to velocity variable $v$ of the gain part is written as a sum of the primary term $\Prim_i$ containing the derivative of the input function $f$ and the reminder $\R_i$,
	\begin{equation}\label{der Q+ wrt vi}
\partial_{v_i} 	Q^+(f,g)(v,I) = \Prim_i + \R_i,
	\end{equation}
where the primary term is pointwise  bounded as follows,
	\begin{equation}\label{th eq prim vel}
		|\Prim_i| \leq  2   \, Q^+\left( \left| \nabla_{v}  f \right|,  \left| g \right|  \right)(v,I), \quad \text{for all} \ i=1,2,3,
	\end{equation}
while the reminder term can be bounded in $L^2(\bRfp)$ as
			\begin{equation}\label{th eq rem vel}
\left\| 	\R_i  \right\|_{L^2}	 \leq 
	\Ccmvel  \| b \|_{{L^{1}}}   	\| f \|_{L^{2}_{(\g+3/2)^+}}\| g \|_{L^{2}_{(\g+3/2)^+}}, \quad \text{for all} \ i=1,2,3,
	\end{equation}
	with  the constant $\Ccmvel>0$ given at the end of the proof.
\end{proposition}

\begin{proof} The proof follows the three main steps:
	
\emph{Step 1: Kernel form of the gain operator and differentiation.} The gain term as defined in \eqref{coll gain operator} depends on the velocity variable through, among other terms, the angular part of the kernel $b(\hat{u}\cdot\sigma)$. Since this term is assumed to be in $ L^{1}(\bSp)$, to avoid differentiation with respect to the velocity variable, it is essential to write the strong form of the gain operator in a kernel form. Differentiation naturally leads to the two terms: the primary term containing the derivative of the input function $f$ and a reminder.
   
\emph{Step 2: Point-wise estimate on the primary term.} 

\emph{Step 3: Estimate on the reminder in $L^2$.} 
  
\subsubsection*{Step 1}	With the aim   to write the gain operator  $Q^+$  from \eqref{coll gain operator} in a kernel form,  perform the change of variables $(v_*,I_*) \mapsto (v'_*,I'_*)$, for the fixed variables $(\param)$. The Jacobian of this transformation can be computed using  similar techniques to those in \cite{Brull-Comp-2} or in Lemma 5.1 from \cite{MC-Alonso-Lp}, and is given by
\begin{equation*}
\left| \frac{\partial (v'_*,I'_*)}{\partial (v_*, I_*)} \right| = \frac{(1-r)(1-R)}{2^3}.
\end{equation*}
Now express all  the variables appearing in \eqref{coll gain operator} in terms of the new variables and   specify a new domain of integration. To that aim,  the collision transformation \eqref{coll rules} is used. Firstly,
\begin{equation*}
	E= \frac{I'_*}{(1-r)(1-R)}, \quad I'=\frac{r}{(1-r)}I'_*, \quad v' = v'_* + 2 \sqrt{\frac{R \, I'_*}{m(1-r)(1-R)}} \ \sigma.
\end{equation*}
Now,   define $w=v-v'_*$, and relate relative velocity $u$ to $w$ via the rule for $v'_*$, namely, 
\begin{equation}\label{u->w}
w=v-v'_*, \quad u=2\left(w -\sqrt{\frac{RE}{m}} \sigma\right) \quad \Rightarrow \quad \hat{u}\cdot\sigma = \frac{w\cdot\sigma-\sqrt{\frac{R I'_*}{m(1-r)(1-R)}}}{\sqrt{|w|^2-2 \sqrt{\frac{R  I'_*}{m(1-r)(1-R)}} \, w \cdot\sigma + \frac{R  I'_*}{m(1-r)(1-R)} } },
\end{equation}
and express $I_*$,
\begin{equation}\label{Is as I}
	\begin{split}
		I_* &= E - \frac{m}{4}|u|^2 -I
	=  \frac{I'_*}{(1-r)} - m |w|^2+2 \sqrt{\frac{m R  I'_*}{(1-r)(1-R)}} \ w \cdot\sigma   -I.
	\end{split}
\end{equation}
Then, the original domain $\left\{   (v_*,I_*,\param)\in \bRfp \times \domparam \right\}$ from \eqref{coll gain operator} changes to
\begin{equation}\label{domain prime}
	\D = \Bigg\{  (v'_*,I'_*,r, {R}, \sigma): \ v'_* \in \mathbb{R}^3,\ I'_*\in\mathbb{R}_+, \ r,  {R} \in[0,1],
 \ \sigma:  w\cdot\sigma \geq \sqrt{\frac{R I'_*}{m(1-r)(1-R)}} \Bigg\},
\end{equation}
and \eqref{coll gain operator} finally becomes
	\begin{multline}\label{coll gain operator new}
	Q^+(f,g)(v,I) =\int_\D  f(v',I')g(v'_*,I'_*) \left(\frac{I }{I' I'_* }\right)^{\alpha}  (\max\{0,I_*\})^\alpha \\ \times b(\hat{u}\cdot\sigma) \, \E^{\g/2} \, \dpo  \, \md \sigma \, \md  {R} \, \md r \, \md I'_* \, \md v'_*.
\end{multline}
Since the argument of the angular kernel $\hat{u}\cdot\sigma $ still depends on the $v$ variable through $w$ and  $\hat{w}\cdot\sigma$ , the   change of variables $R\mapsto \tilde{R}$ is performed  by relation
\begin{equation}\label{bL1 R tilde change}
	\tilde{R} = \frac{R}{\A (1-R)} \in \mathbb{R}_+,  \quad \text{with}
	\quad \A = \frac{(1-r)m|w|^2}{I'_*},
\end{equation} 
fixing all other variables. Then, the   polar coordinates for $\sigma$ are introduced, with zenith $\hat{w}=\tfrac{w}{|w|}$,  
\begin{equation}\label{sigma polar}
	\sigma = \cos \tilde{\theta} \, \hat{w} + \sin\tilde{\theta} \, \sin\varphi \, \omega_1 +  \sin\tilde{\theta} \, \cos\varphi \, \omega_2, 
\end{equation}
where unitary vectors $\hat{w}$, $\omega_1$ and $\omega_2=\hat{w}\times \omega_1$ form an orthonormal basis. Denoting with $\varphi$ the angle describing the one-dimensional sphere in the plane generated by $(\omega_1,\omega_2)$ orthogonal to $\hat{w}$, the Jacobian of polar coordinates change becomes $\md \sigma = 	\md (\cos\tilde{\theta}) \md \varphi$. Now, the argument of the angular part becomes function of $	\tilde{R}$ and $\cos\tilde{\theta}$,  thus independent of $v$, 
\begin{equation}\label{restr theta tilde}
	\cos \tilde{\theta}  := \hat{w}\cdot\sigma, \quad w = v-v'_*, \quad \text{and} \quad  \hat{u}\cdot\sigma = \frac{\cos\tilde{\theta}-\sqrt{\tilde{R}}}{\sqrt{1-2 \sqrt{\tilde{R}} \, \cos\tilde{\theta} + \tilde{R} } },
\end{equation}
which
 motivates to introduce the notation
\begin{equation}\label{b new}
	b(\hat{u}\cdot\sigma) = \tilde{b}(\cos\tilde{\theta},\tilde{R}).
\end{equation}
Then,  the expression appearing in the integrand of the gain term are expressed as follows,
\begin{equation}\label{variables tilde}
	\begin{split}
	& w=v-v'_*, \quad E = m |w|^2 \left(\frac{1}{A} +  \tilde{R}\right), \quad  I'=\frac{r}{(1-r)}I'_*, \quad v' = v'_* +2 |w| \sqrt{\tilde{R}}  \ \sigma, \\
&	I_* = \frac{I'_*}{(1-r)} +m |w|^2 \left( -1 + 2 \sqrt{\tilde{R}} 
	\, \cos\tilde{\theta} \right), \quad  R = \frac{ \A \, \tilde{R}}{1+\A \tilde{R}},
	\end{split}
\end{equation}
with $\A$ from \eqref{bL1 R tilde change} and $\sigma$ represented as \eqref{sigma polar}. 
Finally, the change of variables \eqref{bL1 R tilde change} brings the change of the domain \eqref{domain prime} to a new domain. From \eqref{restr theta tilde},   the restriction $\hat{u}\cdot\sigma\geq 0$ implies $\cos\tilde{\theta}\geq \sqrt{\tilde{R}}$, which in turn restricts the domain for $\tilde{R}$ from $\mathbb{R}_+$ in \eqref{bL1 R tilde change} to $[0,1]$.  
Thus, the new domain is
\begin{multline}\label{domain Dt}
	\Dt = \Bigg\{  (v'_*,I'_*,r,\tilde{R}, \varphi, \cos\tilde{\theta}): \ v'_* \in \mathbb{R}^3,\ I'_*\in\mathbb{R}_+, \ r, \tilde{R} \in[0,1],
	 \\ (\cos\tilde{\theta},\varphi) :   \varphi \in [0, 2\pi), \, \sqrt{\tilde{R}}\leq \cos\tilde{\theta}\leq 1 \Bigg\},
\end{multline}
and the gain operator becomes
	\begin{multline}\label{coll gain operator tilde}
	Q^+(f,g)(v,I) = 2^3 \int_{	\Dt}  f(v',I') \, g(v'_*,I'_*) \left(\frac{I }{I' I'_* }\right)^{\alpha}  \tilde{b}(\cos\tilde{\theta},\tilde{R})    \frac{I'^{\,\g/2-1}_*}{(m(1-r))^{\g/2}} \\ \times 	\hag(I_*, w, r, R) \, 
	\md (\cos\tilde{\theta}) \, \md \varphi \, \md \tilde{R} \, \md r \, \md I'_* \, \md v'_*,
\end{multline}
where 
\begin{equation}\label{h}
	\hag(I_*, w, r, R) =( \max\{0, I_*\})^\alpha \, \dpo (1-R)^{1-\g/2} m |w|^2,
\end{equation}
with $I_*$, $w$, and $R$ as defined in \eqref{variables tilde}.

Thus, the velocity dependence is on $f$ through $v'$ and  on $	\hag$, so taking the derivative of $Q^+$ from \eqref{coll gain operator tilde} with respect to $v_i$ leads naturally to two terms as in \eqref{der Q+ wrt vi}. We consider them separately.

\subsubsection*{Step 2}  The primary term $\Prim_i$ is, for each $i=1,2,3$, given by
	\begin{multline}\label{coll gain operator tilde 1}
\Prim_i = 2^3 \int_{	\Dt} \left( \partial_{v_i}  f(v',I') \right) \, g(v'_*,I'_*) \left(\frac{I }{I' I'_* }\right)^{\alpha}  \tilde{b}(\cos\tilde{\theta},\tilde{R})    \frac{I'^{\,\g/2-1}_*}{(m(1-r))^{\g/2}} \\ \times 	\hag(I_*, w, r, R) \, 	\md (\cos\tilde{\theta}) \, \md \varphi \, \md \tilde{R} \, \md r \, \md I'_* \, \md v'_*.
\end{multline}
The chain rule and the Cauchy-Schwarz inequality imply
\begin{equation}\label{part f}
 \partial_{v_i}  f(v',I')  = \sum_{j=1}^3  \partial_{v'_j}  f(v',I')  \frac{\partial v'_j}{\partial v_i}  \leq \left| \nabla_{v'} f(v',I') \right| \, \left| \partial_{v_i} v' \right|.  
\end{equation}
Next, to compute derivative of $v'$ with respect to each component  $v_i$, the expression \eqref{variables tilde} is used, with   $\sigma$  expressed in terms of polar coordinate system \eqref{sigma polar} that depends on $v$ through $w$,
\begin{equation}\label{part der v' wrt v}
\frac{\partial v'}{\partial v_i} = 2 \sqrt{\tilde{R}} \, \frac{\partial( 	|w|	\sigma )}{\partial w_i}.
\end{equation}
From \eqref{sigma polar},
\begin{equation}\label{w s j}
|w|	\sigma = \cos \tilde{\theta} \,  {w} + \sin\tilde{\theta} \, \sin\varphi \, |w| \, \hat{\xi} +  \sin\tilde{\theta} \, \cos\varphi \,\left( w \times  \hat{\xi} \right),
\end{equation}
for any $\hat{\xi}$ belonging to the plane orthogonal to $w$. To compute the derivative in \eqref{part der v' wrt v} with respect to $w_i$, we will choose a  specific $\hat{\xi}$ for each $i$. This is allowed because we are estimating each $\Prim_i$ independently.  Namely,  for each $i$, we choose $\hat{\xi}$ to be independent of $w_i$. For instance, if $i=1$, we choose $ \xi= (0,-w_3,w_2)$, or if  $i=2$, then $ \xi = (-w_3, 0, w_1)$, or finally for $i=3$, we choose $\xi=(w_1,-w_2,0)$
and $\hat{\xi} = \frac{ {\xi}}{| {\xi}|}$. 
 In such a way,
\begin{equation*}
	\frac{\partial \left( |w|	\sigma\right)}{\partial w_i} = \cos\tilde{\theta} \left( \partial_{w_i} w \right) + \sin\tilde{\theta} \, \sin\varphi \, \frac{w_i}{|w|} \, \hat{\xi} +  \sin\tilde{\theta} \, \cos\varphi \,\left( \left( \partial_{w_i} w \right)  \times  \hat{\xi} \right),
\end{equation*}
and $ \partial_{w_i} w_j = \delta_{ij}$ for each $j=1,2,3$,  where $\delta_{k\ell}$ is the usual Kronecker delta, i.e. equals $1$ for $k=\ell$ and 0 otherwise. Since $i$-th component of $\hat{\xi}$  is chosen zero, vectors $\left( \partial_{w_i} w \right)$, $\hat{\xi}$ and $\left( \partial_{w_i} w \right)  \times  \hat{\xi} $ form an orthonormal basis, leading to 
\begin{equation*}
	\left| 	\frac{\partial \left( |w|	\sigma\right)}{\partial w_i}  \right|^2 =   \cos^2\tilde{\theta} + \sin^2\tilde{\theta} \, \sin^2\varphi \, \frac{w_i^2}{|w|}   +  \sin^2\tilde{\theta} \, \cos^2\varphi \leq 1, \quad \text{for each} \ i=1,2,3.
\end{equation*}
Together with   \eqref{part der v' wrt v}, for \eqref{part f} the last estimate implies 
\begin{equation*}
	\partial_{v_i}  f(v',I')  \leq 2 \sqrt{\tilde{R}} \, \left| \nabla_{v'} f(v',I') \right|, \qquad i=1,2,3.
\end{equation*}
Domain of integration $\Dt$ implies $ \sqrt{\tilde{R}} \leq 1$, leading to the estimate for $\Prim_i$,
	\begin{equation}\label{T1 bound} 
		\begin{split}
 	\left|\Prim_i\right|  &\leq 2^4\int_{	\Dt} \left| \nabla_{v'}  f(v',I') \right| \, \left|g(v'_*,I'_*)\right| \left(\frac{I }{I' I'_* }\right)^{\alpha}  \tilde{b}(\cos\tilde{\theta},\tilde{R})    \frac{I'^{\,\g/2-1}_*}{(m(1-r))^{\g/2}} \\ & \qquad\qquad\qquad \times   \hag(I_*, w, r, R)  \, 	\md (\cos\tilde{\theta}) \, \md \varphi \, \md \tilde{R} \, \md r \, \md I'_* \, \md v'_*
 	\\
 	&= 2  \, Q^+\left( \left| \nabla_{v}  f \right|,  \left| g \right|  \right)(v,I),
 	\end{split}
\end{equation}
which is exactly the statement \eqref{th eq prim vel}.
\subsubsection*{Step 3} The final goal is to  estimate   the reminder term, defined for $\alpha>0$,
	\begin{multline}\label{coll gain operator tilde 2}
\R_i = 2^3 \int_{	\Dt}  f(v',I') \, g(v'_*,I'_*) \left(\frac{I }{I' I'_* }\right)^{\alpha}  \tilde{b}(\cos\tilde{\theta},\tilde{R})    \frac{I'^{\,\g/2-1}_*}{(m(1-r))^{\g/2}} \\ \times \partial_{v_i} \left(	\hag(I_*, w, r, R) \right) \, \dstartilde.
\end{multline}
Taking derivative of $\hag(I_*, w, r, R)$  from \eqref{h}, with the arguments given in \eqref{variables tilde}, implies
\begin{equation*}
\partial_{v_i}  	\hag  =   \kk_i \,	\hag, \quad \text{for} \  \alpha>0,
\end{equation*}
with
\begin{equation}
	\kk_i = \frac{2 w_i}{|w|^2} \left(  \frac{\alpha}{ I_*} m |w|^2 \left( -1 + 2 \sqrt{\tilde{R}} \cos\tilde{\theta}  \right) + \frac{3}{2} - \left( 2 \alpha +\frac{5}{2} - \frac{\g}{2} \right) R   \right),
\label{add 1}
\end{equation}
where   $I_*$ and $R$   are to be understood as functions of variables of integration as given in \eqref{variables tilde}. 
Now we undo all the change of variables and return to the initial ones appearing in \eqref{coll gain operator}, 
	\begin{equation}\label{coll gain operator 2}
\R_i =\istarintro  f(v',I')\,g(v'_*,I'_*) \left(\frac{I I_*}{I' I'_* }\right)^{\alpha}      b(\hat{u}\cdot\sigma) \, \E^{\g/2} \, \kk_i \, \dpo  \, 	\dstar,
\end{equation}
with $\kk_i$ now given by 
\begin{equation}
\kk_i	= \frac{2 w_i}{|w|^2} \left(  \frac{\alpha}{  I_*} \left(R E - \frac{m}{4}|u|^2\right)  + \frac{3}{2} - \left( 2 \alpha +\frac{5}{2} - \frac{\g}{2} \right) R   \right), \quad w = v-v'_*. \label{add 2}
\end{equation}
 Note that $\R_i$ looks similar to $Q^+$ in \eqref{coll gain operator}, with the addition of $\kk_i$ in the kernel. The function $\kk_i$ has a singularity in $w$ and $I_*$ and thus requires  extra effort. We would need to look for an estimate of its weak form and adapt ideas of $L^p$-estimates from \cite{MC-Alonso-Lp}. First, written in a weak form, for some suitable test function $\chi$, using an  interchange of prime and non-prime variables, 
  	\begin{multline*}
  \int_{ \bRfp }	\R_i \, \chi(v,I) \, \md v \, \md I 
  \\ =\iallintro  f(v,I) \, g(v_*,I_*)   \,   b(\hat{u}\cdot\sigma) \, \E^{\g/2} \, \kk'_i \, \chi(v',I')  \, \dpo  \, 	\dall.
  \end{multline*}
Now, let us estimate $\kk'_i$, firstly only using the collision rules \eqref{coll rules},
\begin{equation*}
	\begin{split}
	|\kk'_i| &\leq  \frac{2 }{|w'|} \left(  \frac{\alpha}{ (1-r)(1-R)E } \left(R E + \frac{m}{4}|u|^2\right)  + \frac{3}{2} + \left( 2 \alpha +\frac{5}{2} + \frac{\g}{2} \right) \frac{m}{4} \frac{|u|^2}{E}   \right)
	\\
& 	\leq  \frac{2 }{|w'|} \left(  \frac{ \alpha }{ (1-r)(1-R)E } \left(R E + \frac{m}{4}|u|^2\right)  + \left( 2 \alpha + 4 + \frac{\g}{2} \right) \right)
	\end{split}
\end{equation*}
and secondly, the domain $\bSp$ for $\sigma$ which implies 
\begin{equation*}
|w'|^2 = \left| \frac{1}{2} u' + \sqrt{\frac{R'E}{m}}\sigma' \right|^2  = \frac{1}{2}  \left|\sqrt{\frac{RE}{m}} \sigma + |u| \hat{u} \right|^2 \geq \frac{2}{m} \left(RE + \frac{m}{4}|u|^2\right),
\end{equation*}
and allows to conclude
\begin{equation*}
	|\kk'_i| \leq  \sqrt{2 m} \left(2\alpha + 4 + \frac{\g}{2}\right)  \left(  \frac{\alpha}{ (1-r)(1-R) }  + \frac{1}{\sqrt{R}}\right) \frac{1}{\sqrt{E}}.
\end{equation*}
H\"older inequality applied to the integral with respect to $(v_*,I_*)$ implies, for some weight $s\geq0$ to be specified later,
	\begin{multline*}
	\int_{ \bRfp }	\R_i \, \chi(v,I) \, \md v \, \md I 
	\\ \leq
	\| g \|_{L^2_s} \ \left\| \ \  \int_{\bRfp}   |f(v,I)|    \,  \E^{(\g-1)/2} \ \mS_1(\chi)(v,I, v_*, I_*) \,\md v \,\md I \right\|_{L^2_{-s}(\md v_* \md I_*)},
\end{multline*}
where $\mS_1$  is the averaging operator with the constant $\rho_1$,  defined by $\mS^\psi$ and $\rho^\psi$  from   \eqref{S operator gen} and \eqref{rho gen} for $\psi = \left( \frac{1}{(1-r)(1-R)} + \frac{1}{\sqrt{R}}\right)$.
Applying Cauchy-Schwarz  inequality in $(v,I)$,
	\begin{multline*}
	\int\limits_{(v,I)\in \bRfp }	\R_i \, \chi(v,I) \, \md v \, \md I 
	\\ \leq
	\| f \|_{L^2_s}	\| g \|_{L^2_s}   \sup_{(v_*,I_*)} \left\| \mS_1(\chi)  \right\|_{L^2(\md v  \,\md I)}   \ \sup_{(v,I)}  \Big\| \    \E^{(\g-1)/2} \la v, I \ra^{-s}  \Big\|_{L^2_{-s}(\md v_* \md I_*)}.
\end{multline*}
Appendix Lemma \ref{S estimate L1 gen} and the constant computed in \eqref{constant Ca} imply the final estimate \eqref{th eq rem vel} on the reminder term with the constant
	\begin{equation*}
\Ccmvel=  \sqrt{2 } \left(2\alpha + 4 + \frac{\g}{2}\right)   \rho_1 \, \cb_{(\g -1)/2},
\end{equation*}
with  $\cb_{(\g -1)/2}$ is given in  \eqref{constant Ca} and $\rho_1$ finite for $\alpha>0$.
\end{proof}

\section{Smoothing properties of the gain operator: Proof of Theorem \ref{inc-reg}}

\subsection{Smoothing properties with respect to the velocity variable}\label{Sec: proof smooth v}

\begin{proof}[Proof of Theorem \ref{inc-reg}, Part 1.]
First note that the assumption $f, g \in L^1_{\g}(\bRfp)$ implies $Q^+(f,g) \in L^1(\bRfp)$. Indeed,  the pre-post change of variables,
\begin{multline*}
\left\| Q^+(f,g) \right\|_{L^1(\bRfp)} = 
  \int_{(\bRfp)^2} \int_{  [0,1]^2 \times \bSp}   f(v,I)g(v_*,I_*)    \\ \times b(\hat{u}\cdot\sigma) \, \E^{\g/2} \, \dpo  \, \dall,
\end{multline*}
together with the estimate  on the collision kernel 
\begin{equation}\label{est E}
 \E^{\g/2} \leq \la v, I \ra^\g \la v_*, I_* \ra^\g,
\end{equation}
yield 
\begin{equation*}
\left\| Q^+(f,g) \right\|_{L^1(\bRfp)} \leq \left\| b \right\|_{L^1(\bSp)}  \left\| d_\alpha \right\|_{L^1([0,1]^2)} \left\| f \right\|_{L^1_\g(\bRfp)}   \left\| g \right\|_{L^1_\g(\bRfp)},
\end{equation*}
and therefore  $Q^+(f,g) \in L^1(\bRfp)$. Thus,  we can compute its Fourier transform in both variables $v\in \bR$ and $I\in \bRp$, by using the pre-post change of variables \eqref{weak form pre-post}, for the test function being the Fourier multiplier $\chi(v,I) = e^{-i \xi \cdot v } e^{- i \omega I}$ with $\xi \in \bR$ and $\omega\in \bRr$ being the Fourier variables, 
\begin{multline*}
\reallywidehat{Q^+(f,g)}(\xi,\omega) =  \iall e^{-i \, \xi \cdot \big( \V  + \sqrt{ \frac{RE}{m}} \sigma \big)} e^{-i \, \omega r(1-R)E}   f(v,I)g(v_*,I_*)    \\ \times b(\hat{u}\cdot\sigma) \, \E^{\g/2} \, \dpo  \, \dall.
\end{multline*}
We perform the change of variables 
\begin{equation}\label{Et}
	R \mapsto \tilde{R} = R \Et, \quad  \text{with} \ \Et=\frac{E}{\frac{m}{4}|u|^2} \geq 1.
\end{equation}
Thus, keeping in mind \eqref{fun r R},
\begin{multline}\label{Four-1}
\reallywidehat{Q^+(f,g)}(\xi,\omega) = \iRt e^{-i \, \xi \cdot \big( \V  + \frac{|u|}{2} \sqrt{\tilde{R}}\, \sigma \big)} e^{-i \, \omega r\big(1-\tfrac{\tilde{R}}{\Et}\big)E}   f(v,I)g(v_*,I_*)    \\ \times b(\hat{u}\cdot\sigma) \, \E^{\g/2} \, r^\alpha(1-r)^\alpha \left(1-\tfrac{\tilde{R}}{\Et}\right)^{2\alpha+1} \frac{\sqrt{\tilde{R}}}{\Et^{3/2}} \, \mathds{1}_{\tilde{R}\in[0,\Et]} \, \dRt.
\end{multline}
Define
\begin{multline}\label{F}
	F(v,v_*,\tilde{R},\omega) =  \iI  f(v,I)g(v_*,I_*)  \E^{\g/2} 
		\\ \times 
		 e^{-i \, \omega r\big(1-\tfrac{\tilde{R}}{\Et}\big)E}  r^\alpha(1-r)^\alpha \left(1-\tfrac{\tilde{R}}{\Et}\right)^{2\alpha+1} \frac{\sqrt{\tilde{R}}}{\Et^{3/2}}  (1+\tilde{R}) \, \mathds{1}_{\tilde{R}\in[0,\Et]} \dI,
\end{multline}
so that \eqref{Four-1} can be rewritten in terms of $F$ as follows
\begin{equation*}
\reallywidehat{Q^+(f,g)}(\xi,\omega) = \iv   \frac{	F(v,v_*,\tilde{R},\omega)}{(1+\tilde{R})} \, b(\hat{u}\cdot\sigma)  \, e^{-i \, \xi \cdot \big( \V  + \frac{|u|}{2} \sqrt{\tilde{R}}\, \sigma \big)}   \dv.
\end{equation*}
With the aim of removing the velocity dependence in $b$, one follow the standard arguments \cite{Bob-Exact-solutions-Fourier,BD} and exchange the unitary vectors $\hat{\xi} \leftrightarrow \hat{u}$ with an orthogonal transformation performed in the $\sigma$--integration,
\begin{align}\label{Four-2}
	\reallywidehat{Q^+(f,g)}&(\xi,\omega)  = \iv \frac{	F(v,v_*,\tilde{R},\omega)}{(1+\tilde{R})} \, b(\hat{\xi}\cdot\sigma)\, e^{-i \, \xi \cdot \big( \V   \big)-i \,   \frac{|\xi|}{2} \sqrt{\tilde{R}}\, u \cdot  \sigma }     \,   \dv \nonumber
	\\
	&= \iv \frac{	F(v,v_*,\tilde{R},\omega)}{(1+\tilde{R})} \, b(\hat{\xi}\cdot\sigma)\, e^{-i \, v \cdot\big( \frac{\xi}{2}  +  \frac{1}{2} \sqrt{\tilde{R}}\, |\xi|   \sigma   \big)} e^{-i \, v_* \cdot\big( \frac{\xi}{2}  -  \frac{1}{2} \sqrt{\tilde{R}}\, |\xi|   \sigma   \big)}     \,   \dv \nonumber
	\\ 
	&= \iRts \frac{1}{(1+\tilde{R})} \, b(\hat{\xi}\cdot\sigma)\, 	\reallywidehat{F(\cdot,\cdot,\tilde{R},\omega)}\Big( \tfrac{\xi + \sqrt{\tilde{R}}\, |\xi|   \sigma }{2}, \tfrac{\xi  - \sqrt{\tilde{R}}\, |\xi|   \sigma}{2}   \Big)    \dRts,
\end{align}
where   $\reallywidehat{F(\cdot,\cdot,z)}(x,y)$ shall be understood as the Fourier transform of $F$ with respect to the first two variables evaluated at $(x,y)$ while keeping $z$ fixed, i.e.
\begin{equation*}
\reallywidehat{F(\cdot,\cdot,z)}(x,y) = \!\!\!\int\limits_{ (v,v_*)\in \bR\times\bR } F(x,y,z) \, e^{- i v \cdot x} e^{- i v_* \cdot y}  \md v \, \md v_*, \ \text{for any} \ x, y \in \bR,\, z \in \mathbb{R}^d,\, d\geq1.
\end{equation*}
From \eqref{Four-2}, we apply the Cauchy-Schwarz inequality in $\sigma$ that pulls out $L^2-$norm of the angular kernel $b$. Then, since the Cauchy-Schwarz inequality in $\tilde{R}$ implies
\begin{multline*}
 \int_0^\infty \frac{1}{(1+\tilde{R})} \,  	\reallywidehat{F(\cdot,\cdot,\tilde{R},\omega)}\Big( \tfrac{\xi +\sqrt{\tilde{R}}\, |\xi|   \sigma }{2}, \tfrac{\xi - \sqrt{\tilde{R}}\, |\xi|   \sigma}{2}\Big) \md \tilde{R}
 \\
  \leq \left( \int_0^\infty \frac{1}{\tilde{R}^{1/2}(1+\tilde{R})} \right)^{1/2}  \left(   \int_0^\infty \frac{\tilde{R}^{1/2}}{(1+\tilde{R})} \, \left|  \reallywidehat{F(\cdot,\cdot,\tilde{R},\omega)}\Big( \tfrac{\xi +\sqrt{\tilde{R}}\, |\xi|   \sigma }{2}, \tfrac{\xi - \sqrt{\tilde{R}}\, |\xi|   \sigma}{2}\Big)  \right|^2 \md \tilde{R} \right)^{1/2},
\end{multline*}
denoting 
\begin{equation}\label{Four-3}
	\Fou= \int_{\bSp}  \int_0^\infty \frac{\tilde{R}^{1/2}}{(1+\tilde{R})} \, \left|   \reallywidehat{F(\cdot,\cdot,\tilde{R},\omega)}\Big( \tfrac{\xi +\sqrt{\tilde{R}}\, |\xi|   \sigma }{2}, \tfrac{\xi - \sqrt{\tilde{R}}\, |\xi|   \sigma}{2}\Big) \right|^2 \md \tilde{R}   \, \md \sigma
\end{equation}
the Fourier transform of the gain part \eqref{Four-2} is estimated as
\begin{equation}\label{Q+ Four tr}
\left|	\reallywidehat{Q^+(f,g)}(\xi,\omega)  \right|
	\\
	 \leq  \pi^{1/2} \, \|b\|_{L^2(\bSp)} \Fou^{1/2}.
\end{equation}
In order to proceed with the estimation, we study  $\Fou$ from \eqref{Four-3}. The idea is, as in \cite{BD}, to use $\sigma$-integration in $\bSp$ with $|\xi|$ dependence to pass to the full $\bR$-integration and conveniently extract $|\xi|^2$ factor that will give a proper order of the Sobolev norm in the velocity variable. With respect to \cite{BD}, we need to take care of the extra $\tilde{R}$-integration. First, using the fundamental theorem of calculus,  \eqref{Four-3} can be rewritten as
\begin{equation*}
\Fou=   \int_0^\infty \int_{\bSp}  \int_{|\xi|}^{\infty} \frac{\tilde{R}^{1/2}}{(1+\tilde{R})} \, 	\partial_\eta \left|  \reallywidehat{F(\cdot,\cdot,\tilde{R},\omega)}\Big( \tfrac{\xi + \sqrt{\tilde{R}} \, \eta \, \sigma}{2}, \tfrac{\xi - \sqrt{\tilde{R}} \, \eta \, \sigma}{2}\Big)\right|^2 \, \md \eta \, \md \sigma \,  \md \tilde{R}.
\end{equation*}  
Computing the involved derivative, one can estimate,
\begin{multline*}
\leq  \int_0^\infty \int_{\bSp}  \int_{|\xi|}^{\infty} \frac{\tilde{R}}{(1+\tilde{R})} \,  \left|    \reallywidehat{F(\cdot,\cdot,\tilde{R},\omega)}\Big( \tfrac{\xi + \sqrt{\tilde{R}} \, \eta \, \sigma}{2}, \tfrac{\xi - \sqrt{\tilde{R}} \, \eta \, \sigma}{2}\Big)  \right|  
\\
\times \left| \left( \nabla_2-\nabla_1 \right)  \reallywidehat{F(\cdot,\cdot,\tilde{R},\omega)}\Big( \tfrac{\xi + \sqrt{\tilde{R}} \, \eta \, \sigma}{2}, \tfrac{\xi - \sqrt{\tilde{R}} \, \eta \, \sigma}{2}\Big)   \right| \, \md \eta \, \md \sigma \,  \md \tilde{R},
\end{multline*}
where $\nabla_i$ denotes the gradient with respect to $i$-th variable, $i=1,2$. 
Next, we perform the change of variables $\eta \mapsto \check{\eta} = \sqrt{ \tilde{R}}\,  \eta$,   and extend the space for $\sigma$ to the full sphere $\bS$,
\begin{multline*}
	\Fou
	\leq  \int_0^\infty \int\limits_{\substack{\sigma \in \bS \\ \check{\eta} \in \mathbb{R}: \,   \check{\eta}\geq \sqrt{\tilde{R}}|\xi| } }   \frac{\tilde{R}^{1/2}}{(1+\tilde{R})} \,  \left| \reallywidehat{F(\cdot,\cdot,\tilde{R},\omega)}\Big( \tfrac{\xi + \check{\eta} \, \sigma }{2}, \tfrac{\xi  - \check{\eta}\, \sigma  }{2} \Big) \right| 
	\\
	\times \left| \left( \nabla_2-\nabla_1 \right) \reallywidehat{F(\cdot,\cdot,\tilde{R},\omega)}\Big( \tfrac{\xi + \check{\eta} \, \sigma }{2}, \tfrac{\xi  - \check{\eta}\, \sigma  }{2} \Big)  \right| \, \md \check{\eta} \, \md \sigma \,  \md \tilde{R}.
\end{multline*}
The next step is to combine spherical $\sigma\in\bS$ integration and one-dimensional $\check{\eta}$ integration into an integration over $\bR$-space  with respect to the variable $\overline{\eta}=\check{\eta}\,\sigma$ with the usual Jacobian of the three-dimensional spherical coordinates change $\frac{1}{|\overline{\eta}|^2} \md \overline{\eta} = \md \check{\eta} \, \md \sigma$, so that the last inequality becomes
\begin{equation}\label{Four-4}
	\Fou
	\leq \frac{1}{|\xi|^2} \tilde{\Fou},
\end{equation}
with 
\begin{multline*}
	\tilde{\Fou}
=  \int_0^\infty   \int\limits_{\overline{\eta} \in \bR: |\overline{\eta}|\geq\sqrt{\tilde{R}  } |\xi| } \frac{\tilde{R}^{-1/2}}{(1+\tilde{R})} \,  \left|  \reallywidehat{F(\cdot,\cdot,\tilde{R},\omega)}\Big( \tfrac{\xi+\overline{\eta}}{2}, \tfrac{\xi-\overline{\eta}}{2}\Big) \right| 
\\
\times \left| \left( \nabla_2-\nabla_1 \right)   \reallywidehat{F(\cdot,\cdot,\tilde{R},\omega)}\Big( \tfrac{\xi+\overline{\eta}}{2}, \tfrac{\xi-\overline{\eta}}{2}\Big)     \right| \, \md \overline{\eta} \,   \md \tilde{R}.
\end{multline*}
Returning \eqref{Four-4} into \eqref{Q+ Four tr}, one gets
\begin{multline}\label{Four-5}
\| Q^+(f,g) \|^2_{\dot{H}^{1}_{v}(\bRfp)} = \int_{(\xi,\omega)\in \bRfr} |\xi|^2 \,	|\reallywidehat{Q^+(f,g)}(\xi,\omega)|^2 \ \md \omega \, \md \xi
	\\
	\leq  \pi \, \|b\|_{L^2(\bSp)}^2  \int_{(\xi,\omega)\in \bRfr}   \tilde{\Fou}  \, \md \omega \, \md \xi.
\end{multline}
The next goal is to estimate the integral on the right-hand side of  the last inequality \eqref{Four-5}.
Extending the domain for $\overline{\eta}$ variable, one gets
\begin{multline*}
\int_{(\xi,\omega)\in \bRfr}	\tilde{\Fou} \, \md \omega \, \md \xi 
	\leq   \int_0^\infty  \frac{\tilde{R}^{-1/2}}{(1+\tilde{R})}  \int\limits_{ \substack{(\xi,\omega)\in \bRfr \\ \overline{\eta} \in \bR }} \,  \left|   \reallywidehat{F(\cdot,\cdot,\tilde{R},\omega)}\Big( \tfrac{\xi+\overline{\eta}}{2}, \tfrac{\xi-\overline{\eta}}{2}\Big)   \right| 
	\\
	\times \left| \left( \nabla_2-\nabla_1 \right)   \reallywidehat{F(\cdot,\cdot,\tilde{R},\omega)}\Big( \tfrac{\xi+\overline{\eta}}{2}, \tfrac{\xi-\overline{\eta}}{2}\Big)    \right| \, \md \overline{\eta} \, \, \md \xi \, \md \omega \,  \md \tilde{R}.
\end{multline*}
Now, we change the variables $(\xi, \overline{\eta}) \mapsto (\tilde{\xi},\tilde{\eta}) = \Big(  \frac{\xi+\overline{\eta}}{2},  \frac{\xi-\overline{\eta}}{2}\Big)$ with the  Jacobian $\frac{\partial(\tilde{\xi},\tilde{\eta})}{\partial(\xi, \overline{\eta}) } = 2^{-3}$,
\begin{multline*}
	\int_{(\xi,\omega)\in \bRfr}	\tilde{\Fou} \, \md \omega \, \md \xi 
	\leq 2^3  \int_0^\infty  \frac{\tilde{R}^{-1/2}}{(1+\tilde{R})} 
	\\ \times \int\limits_{ \substack{(\tilde{\xi},\omega)\in \bRfr \\ \tilde{\eta} \in \bR }} \,  \left|   \reallywidehat{F(\cdot,\cdot,\tilde{R},\omega)}(\tilde{\xi},\tilde{\eta})   \right| 
  \left| \left( \nabla_{\tilde{\eta}}-\nabla_{\tilde{\xi}} \right) \reallywidehat{F(\cdot,\cdot,\tilde{R},\omega)}(\tilde{\xi},\tilde{\eta})    \right| \, \md \tilde{\eta} \, \, \md \tilde{\xi} \, \md \omega \,  \md \tilde{R}.
\end{multline*}
Next, keeping integral with respect to $\tilde{R}$ as the outer integral and  performing Cauchy-Schwarz with respect to other variables, one gets, after  Plancherel identity,
\begin{multline}
	\int_{(\xi,\omega)\in \bRfr}	\tilde{\Fou} \, \md \omega \, \md \xi 
	\leq 2^3  \int_0^\infty  \frac{\tilde{R}^{-1/2}}{(1+\tilde{R})} 
%
%
	\\
	\times	 \left( \int_{  \omega \in \bRr} \,  \left\|  \reallywidehat{F(\cdot,\cdot,\tilde{R}, \omega)}  \right\|^2_{L^2(\bR\times\bR)}  \, \md \omega \right)^{1/2} \left( \int_{  \omega \in \bRr} 
	\left\| \left( \nabla_{2}-\nabla_{1} \right) \reallywidehat{F(\cdot,\cdot,\tilde{R}, \omega)}  \right\|^2_{L^2(\bR\times\bR)}  \md \omega \right)^{1/2}   \md \tilde{R}
	\\
	= 2^3 (2\pi)^{6}   \int_0^\infty  \frac{\tilde{R}^{-1/2}}{(1+\tilde{R})} 
	\left( \int_{  \omega \in \bRr} \,  \left\|   F(\cdot,\cdot,\tilde{R},\omega) \right\|^2_{L^2(\bR\times\bR)}  \, \md \omega \right)^{1/2} 
	\\ \times
	\left( \int_{  \omega \in \bRr} 
	\left\| (\cdot_2-\cdot_1)  F(\cdot,\cdot,\tilde{R},\omega)  \right\|^2_{L^2(\bR\times\bR)}  \md \omega \right)^{1/2}   \md \tilde{R}
	\\
		= 2^3 (2\pi)^{6}   \int_0^\infty  \frac{\tilde{R}^{-1/2}}{(1+\tilde{R})} 
  \left\|   F(\cdot,\cdot,\tilde{R},\cdot) \right\|_{L^2(\bR\times\bR\times\bRr)} 
	\left\| (\cdot_2-\cdot_1)  F(\cdot,\cdot,\tilde{R},\cdot)  \right\|_{L^2(\bR\times\bR\times\bRr)}   \md \tilde{R}. \label{est F 4}
\end{multline}
In order to proceed, we need to suitably estimate $F$.  First, using Lemma \ref{Lemma F est}, \eqref{F est lemma} applied to $x=\omega \,\big(1-\tfrac{\tilde{R}}{\Et}\big)E$ implies, 
 \begin{multline}\label{F est 1}
	\left| \int_{r\in[0,1]} e^{-i  \, \omega \,\big(1-\tfrac{\tilde{R}}{\Et}\big)E \, r} \,  r^\alpha(1-r)^\alpha \md r  \right| \leq  2 \min \left\{ 1, \frac{1}{|\omega|  \,\big(1-\tfrac{\tilde{R}}{\Et}\big)E} \right\}
	\\ \leq 2 \min \left\{ 1, \frac{1}{\left(|\omega|  \,\big(1-\tfrac{\tilde{R}}{\Et}\big)E\right)^s} \right\},
\end{multline}
for any $0\leq s \leq 1$ to be chosen later.
Using the properties of $\min$ function, the region for $\omega$ can be split, to obtain the  final estimate
\begin{equation*}
		\left| \int_{r\in[0,1]} e^{-i \, \omega \, r\big(1-\tfrac{\tilde{R}}{\Et}\big)E}  r^\alpha(1-r)^\alpha \md r  \right| \leq  2 \left( \mathds{1}_{|\omega|\leq 1}  +  \frac{1}{\Big(\omega \big(1-\tfrac{\tilde{R}}{\Et}\big)E\Big)^s} \mathds{1}_{|\omega| \geq 1}\right).
\end{equation*}
Second, we estimate,
\begin{equation*}
	\left(1-\tfrac{\tilde{R}}{\Et}\right)^{2\alpha+1 - s } \frac{\sqrt{\tilde{R}}}{\Et^{3/2}}  (1+\tilde{R}) \, \mathds{1}_{\tilde{R}\in[0,\Et]}    \leq \frac{ (1+\tilde{R})}{\Et} \mathds{1}_{\tilde{R}\in[0,\Et]}  \leq 2,
\end{equation*}
since $\Et \geq 1$, by \eqref{Et}. Thus,
\begin{multline}\label{est F}
\left| 	F(v,v_*,\tilde{R},\omega)  \right| \leq 4  \mathds{1}_{|\omega| \leq 1} \ien  |f(v,I)| \, |g(v_*,I_*)|  \E^{\g/2} \md I_* \, \md I
\\  + \frac{4 }{|\omega|^s  \, m^s} \mathds{1}_{|\omega| \geq 1} \ien  |f(v,I)| \, |g(v_*,I_*)|  \E^{\g/2-s} \md I_* \, \md I.
\end{multline}
For the first term in last inequality, the estimate on the collision kernel \eqref{est E} and Cauchy-Schwarz inequality in $I$ and $I_*$ imply
\begin{multline*}
\ien  |f(v,I)| \, |g(v_*,I_*)|  \E^{\g/2} \md I_* \, \md I 
\\
 \leq \left( \int_{I\in \bRp} |f(v,I)|^2   \la v, I \ra^{2k} \, \md I \right)^{1/2} \left( \int_{I_*\in \bRp} |g(v_*,I_*)|^2   \la v_*, I_* \ra^{2k} \,  \md I_*  \right)^{1/2} \int_{I\in \bRp} \la v, I \ra^{2(\gamma - k)} \md I,
\end{multline*}
for $k$ sufficiently large enough ensuring the finiteness of the term
\begin{equation}\label{c1 tilde}
\int_{I\in \bRp} \la v, I \ra^{2(\gamma - k)} \md I \leq \int_{I\in \bRp} \left(1+\tfrac{I}{m}  \right)^{\gamma - k} \md I   =: \cI, \quad k > \gamma+1.
\end{equation}
For the second term, restrict $s$ so that $-1 < \gamma/2-s \leq 0$, 
\begin{equation*}
\E^{\g/2 - s } \leq \left(\frac{ I +I_*}{m}\right)^{\g/2-s},
\end{equation*}
and, on the other side, for any $k\geq 0$, since $1+\tfrac{I+I_*}{m}\leq \la v, I\ra^{2}\la v_*, I_*\ra^{2}$,
\begin{equation*}
	1 \leq \frac{\la v, I\ra^{k}\la v_*, I_*\ra^{k}}{\left(1+\tfrac{I+I_*}{m}\right)^{k/2}}.
\end{equation*}
Thus,
\begin{multline*}
 \ien  |f(v,I)| \, |g(v_*,I_*)|  \E^{\g/2-s} \md I_* \, \md I    
 \\
 \leq  \ien  |f(v,I)|   \la v, I \ra^{k}  |g(v_*,I_*)|  \la v_*, I_* \ra^{k}  \frac{\left(\frac{ I +I_*}{m}\right)^{\g/2-s}}{\left(1+\tfrac{I+I_*}{m}\right)^{k/2}} \md I_* \, \md I.
\end{multline*}
The Cauchy-Schwarz inequality in $(I,I_*)$ implies
\begin{multline}\label{F est 2}
 \ien  |f(v,I)| \, |g(v_*,I_*)|  \E^{\g/2-s} \md I_* \, \md I    
\\
\leq    \left( \int_{I\in \bRp} |f(v,I)|^2   \la v, I \ra^{2k} \, \int_{I_*\in \bRp}   \frac{\left(\frac{ I +I_*}{m}\right)^{\g/2-s}}{\left(1+\tfrac{I+I_*}{m}\right)^{k/2}}  \md I_* \, \md I \right)^{1/2} 
\\ \times \left( \int_{I_*\in \bRp} |g(v_*,I_*)|^2   \la v_*, I_* \ra^{2k} \, \int_{I\in \bRp}   \frac{\left(\frac{ I +I_*}{m}\right)^{\g/2-s}}{\left(1+\tfrac{I+I_*}{m}\right)^{k/2}}  \md I \,   \md I_*  \right)^{1/2}.
\end{multline}
Since $\g/2-s\leq 0$, $k\geq 0$,
\begin{equation*}
\frac{\left(\frac{ I +I_*}{m}\right)^{\g/2-s}}{\left(1+\tfrac{I+I_*}{m}\right)^{k/2}}  \leq \frac{\left(\frac{ I }{m}\right)^{\g/2-s}}{\left(1+\tfrac{I}{m}\right)^{k/2}}, \quad \text{for any } \ I, I_* \in \bRp,
\end{equation*}
the $I$--integral in \eqref{F est 2} is bounded by the constant 
\begin{equation}\label{c2 tilde}
\cIIs =  \int_{I\in \bRp}   \frac{\left(\frac{ I}{m}\right)^{\g/2-s}}{\left(1+\tfrac{I}{m}\right)^{k/2}} \,   \md I, \quad k>\gamma+2-2s, \quad -1 < \g/2 - s \leq 0.
\end{equation}
Therefore, \eqref{F est 2} is estimates as
\begin{multline}\label{F est 3}
	\ien  |f(v,I)| \, |g(v_*,I_*)|  \E^{\g/2-s} \md I_* \, \md I    
	\\
	\leq  \cIIs  \left( \int_{I\in \bRp} |f(v,I)|^2   \la v, I \ra^{2k} \, \md I \right)^{1/2} 
\left( \int_{I_*\in \bRp} |g(v_*,I_*)|^2   \la v_*, I_* \ra^{2k}  \,   \md I_*  \right)^{1/2}
\\
=  \cIIs  \| f(v,\cdot)\|_{L^2_{k}(\bRp)}  \| g(v_*,\cdot)\|_{L^2_{k}(\bRp)}\,,
\end{multline}
which, in turn, leads to the following estimate for \eqref{est F} 
\begin{equation*}
	\left| 	F(v,v_*,\tilde{R},\omega)  \right| \leq 4 \,   \max\{\cI,  \cIIs\} 
	\Bigg( \mathds{1}_{|\omega| \leq 1} 
	+ \frac{1}{|\omega|^s  \, m^s} \mathds{1}_{|\omega| \geq 1} \Bigg)  \| f(v,\cdot)\|_{L^2_{k}(\bRp)}  \| g(v_*,\cdot)\|_{L^2_{k}(\bRp)},
\end{equation*}
where the constants are given in \eqref{c1 tilde} and \eqref{c2 tilde}. 
Thus, the following  $\tilde{R}$--uniform estimate is obtained,
\begin{equation*}
	\left\|   F(\cdot,\cdot,\tilde{R},\cdot) \right\|_{L^2(\bR\times\bR\times\bRr)}   \leq c_\g \,  \| f\|_{L^2_k(\bRfp)} \,  \| g\|_{L^2_k(\bRfp)},
\end{equation*}
where the constant $c_\g $ is given by 
\begin{equation}\label{cg}
	\begin{split}
	c_\g  & = \left(4 \max\{\cI,  \cIIs\} \right)^2  \int_{\omega\in\bRr}	\Bigg( \mathds{1}_{|\omega| \leq 1} 
	+ \frac{1}{|\omega|^{2s}  \, m^{2s}} \mathds{1}_{|\omega| \geq 1} \Bigg) \md \omega 
\\&	= 2 \left(4 \max\{\cI,  \cIIs \} \right)^2    \left( 1 + \frac{m^{-2s}}{2s-1} \right),
 	\end{split}
\end{equation}
and the following   choices for $k$ and $s$
\begin{equation*}
k>1+\gamma \qquad \text{and} \qquad
\bigg\{\begin{array}{clc}
s = 1 &\text{if} & 0<\gamma\leq 2 \\
1/2<s<1 & \text{if} & \gamma=0\,.
\end{array}
\end{equation*}
 Finally, in order to estimate \eqref{est F 4}, we notice that 
\begin{equation*}
	|v-v_*| \leq 2 \, \la v,I \ra  \la v_*,I_* \ra,
\end{equation*}
and associate the brackets to the corresponding input functions $f$ and $g$.  This implies that
\begin{equation*}
		\left\| (\cdot_2-\cdot_1)  F(\cdot,\cdot,\tilde {R},\cdot)  \right\|_{L^2(\bR\times\bR\times\bRr)}  \leq 2 \, c_\g \,  \| f\|_{L^2_{k+1}(\bRfp)} \,  \| g\|_{L^2_{k+1}(\bRfp)}.
\end{equation*}
Thus, for any $0\leq \g \leq 2$, with the aforementioned choice for $k$ and $s$,
\begin{multline*}
	\left\|   F(\cdot,\cdot,\tilde{R},\cdot) \right\|_{L^2(\bR\times\bR\times\bRr)} 
	\left\| (\cdot_2-\cdot_1)  F(\cdot,\cdot,\tilde{R},\cdot)  \right\|_{L^2(\bR\times\bR\times\bRr)}\\  
	\leq 2 \, c_\g^2  \| f\|^2_{L^2_{k+1}(\bRfp)} \,  \| g\|^2_{L^2_{k+1}(\bRfp)},
\end{multline*}
which leads to
\begin{equation*}
	\int_{(\xi,\omega)\in \bRfr}	\tilde{\Fou} \, \md \omega \, \md \xi 
	\leq
	 2^4 (2\pi)^{6}  \pi \, c_\g^2 \,  \| f\|^2_{L^2_{k+1}(\bRfp)} \,  \| g\|^2_{L^2_{k+1}(\bRfp)}.
\end{equation*}
This proves the statement \eqref{eq Th smooth vel}, following \eqref{Four-5}, with the generic constant now given by
\begin{equation}\label{C smooth vel}
\Csvel =2^2 (2\pi)^{3}  \pi \, c_\g,
\end{equation}
and $c_\g$ from \eqref{cg}.

\end{proof}

\subsection{Smoothing properties with respect to the internal energy variable}\label{Sec: proof smooth I}

\begin{proof}[Proof of Theorem \ref{inc-reg}, Part 2.] We start with the gain operator in the form \eqref{coll gain operator new}. Noting that $I_*$ is actually a function of $I$ as described in \eqref{Is as I}
	\begin{equation*}
\partial_I  (\max\{0,I_*\})^\alpha  = - \alpha \max\{0,I_*^{\alpha-1}\}, \quad \text{for} \ \alpha>0,
	\end{equation*}
the following derivative can be computed
\begin{equation*}
	\partial_I \left(I I_*\right)^\alpha = \alpha I^\alpha   (\max\{0,I_*\})^\alpha  \left(\frac{1}{I} - \frac{1}{I_*}\right) =:  I^\alpha   (\max\{0,I_*\})^\alpha  \kk, \quad \kk:=\alpha \left(\frac{1}{I} - \frac{1}{I_*}\right).
\end{equation*}
Returning to the original variables, we get exactly \eqref{coll gain operator} with the addition of the term $\kk$,
		\begin{multline*}
\RI :=\partial_{I}		Q^+(f,g)(v,I) =\istarintro  f(v',I')g(v'_*,I'_*) \left(\frac{I }{I' I'_* }\right)^{\alpha}  (\max\{0,I_*\})^\alpha \\ \times b(\hat{u}\cdot\sigma) \, \E^{\g/2} \, \kk \, \dpo \, \dstar.
	\end{multline*}
Then, its weak form, for some suitable test function $\chi$, using an  interchange of prime and non-prime variables, reads
\begin{multline*}
	\int_{ \bRfp } \test \, \RI \, \chi(v,I) \, \md v \, \md I 
	=\iallintro  f(v,I) \, g(v_*,I_*)   \,   b(\hat{u}\cdot\sigma) \, \E^{\g/2} \, \\ \times \kk'\, \testp \chi(v',I')  \, \dpo  \, 	\dall.
\end{multline*}
Then, since
\begin{equation*}
  \kk' \testp =  \alpha   \left(\frac{1}{I'} - \frac{1}{I'_*}\right) \testp = \alpha \left( r^{\delta - 1 } - \frac{r^\delta}{(1-r)} \right) \left(1-R\right)^{\delta-1} E^{\delta -1},
\end{equation*}
the last equality can be written as
\begin{multline*}
	\int_{ \bRfp } \test \, \RI \, \chi(v,I) \, \md v \, \md I 
\\	 \leq m^{\delta-1} \int_{(\bRfp)^2} | f(v,I) | \, | g(v_*,I_*) |  \,  \E^{\g/2+\delta-1} 	\mS_2(|\chi|) \,	\dvvs,
\end{multline*}
where $\mS_2$  is the averaging operator with the constant $\rho_2$,  defined by $\mS^\psi$ and $\rho^\psi$  from   \eqref{S operator gen} and \eqref{rho gen} for   $\psi= \alpha \left( r^{\delta - 1 }   + {r^\delta}{(1-r)^{-1}} \right) \left(1-R\right)^{\delta-1}$. Then using the estimate on $L^2$-norm of $	\mS_2(\chi) $ given in  \eqref{S estimate L1 gen} with   the constant $\rho_2$ finite for $ \alpha+\delta>\frac{1}{2}$, 
the Cauchy-Schwartz inequality in $(v,I)$, and then in $(v_*,I_*)$ imply
\begin{multline*}
	\int_{\bRfp} \test \, \RI \, \chi(v,I) \, \md v \, \md I 
		 \leq   m^{\delta-1}   \rho_2 \,   \| b \|_{\Ls(\bSp)}   \| \chi\|_{L^{2}}  \\ \times \int_{\bRfp} \, |g(v_*,I_*)|  \left( \int_{\bRfp}  |f(v,I)|^2   \E^{\g+2\delta-2}  \md v\, \md I \right)^{1/2}	 	\md v_* \md I_*\\
		 	 \leq    m^{\delta-1}   \rho_2 \,   \| b \|_{\Ls(\bSp)}   \| \chi\|_{L^{2}}\| g\|_{L^2_k}  \left( \int_{(\bRfp)^2}  |f(v,I)|^2   \E^{\g+2\delta-2} \la v_*, I_*\ra^{-2k} \,  \dvvs \right)^{1/2}
		 	 \\
		 	 \leq  m^{\delta-1}   \rho_2 \,   \| b \|_{\Ls(\bSp)}   \| \chi\|_{L^{2}}   \| g\|_{L^2_{(\g+2\delta+1/2)^+}} \mathcal{C}_{\g/2+\delta-1}  \| f\|_{L^2_{(\g+2\delta+1/2)^+}}.
\end{multline*}
This leads to 
	\begin{multline}
	\int_{\bRfp} \test \, \left(\partial_{I}		Q^+(f,g)(v,I) \right) \, \chi(v,I) \, \md v \, \md I 
	\\
	\leq  m^{\delta-1}   \rho_2 \,   \| b \|_{\Ls(\bSp)}\  \mathcal{C}_{\g/2+\delta-1}  \| f\|_{L^2_{(\g+2\delta+1/2)^+}}    \| g\|_{L^2_{(\g+2\delta+1/2)^+}}    \| \chi\|_{L^{2}},
\end{multline}
where  the constant $\mathcal{C}$ is given in \eqref{constant Ca}. This concludes the proof with a constant
\begin{equation}\label{C smooth I}
	\Csen =  m^{\delta-1}   \rho_2 \,    \mathcal{C}_{\g/2+\delta-1}. 
\end{equation}

\end{proof}

\section{Application to the Boltzmann equation: Proof of Theorem \ref{Th regularity}}\label{Sec: proof reg}

\begin{proof}[Proof of Theorem \ref{Th regularity}, Part 1.]
Without loss of generality, by density arguments, 
the angular part of the collision kernel $b \in L^1(\bSp)$ is split into $b_2 \in L^2(\bSp)$ and $\tilde{b} \in L^1(\bSp)$ in the following way,
\begin{equation}
b(\hat{u}\cdot\sigma) = b_2(\hat{u}\cdot\sigma)  + \tilde{b}(\hat{u}\cdot\sigma), \quad \text{with} \quad  \|\tilde{b}\|_{L^1(\bSp)} \leq \varepsilon,
\end{equation}
with $\varepsilon$ to be chosen later. Then, the gain part is split into the two terms for each of those collision kernels,
\begin{equation}
	Q_b^+(f,g)(v,I) = Q_{b_2}^+(f,g)(v,I)   + Q_{\tilde{b}}^+(f,g)(v,I),
\end{equation}
with subscripts indicating the corresponding angular part. The idea is to apply Theorem \ref{inc-reg} on $Q_{b_2}^+$ and Theorem \ref{inc-reg-b1} on $Q_{\tilde{b}}^+$. Recalling the Boltzmann equation \eqref{BE}, 
derivative 
with respect to $i$-th component of velocity variable, $v_i$,  reads
\begin{equation}\label{BE der vi}
\partial_t \partial_{v_i} f = \partial_{v_i} Q_{b_2}^+(f,f)(v,I)  +  \partial_{v_i} Q_{\tilde{b}}^+(f,f)(v,I) - 	\nu[f] \partial_{v_i} f  - f \partial_{v_i} 	\nu[f].
\end{equation}
With the notation \eqref{H1v}, multiplying \eqref{BE der vi} by $\partial_{v_i} f$ and integrating with respect to  $(v,I) $ implies
\begin{multline*}
\frac{1}{2} \,\frac{\md}{\md t}\| f \|_{\dot{H}^{1}_{v_i}}^2 = 	\int_{ \bRfp } \partial_{v_i} f \  \partial_{v_i} Q_{b_2}^+(f,f)  \, \md v \, \md
+   	\int_{ \bRfp } \partial_{v_i} f\, \partial_{v_i} Q_{\tilde{b}}^+(f,f) \, \md v \, \md I
\\ - 	\int_{ \bRfp }	\nu[f] \, |\partial_{v_i} f|^2 \, \md v \, \md I
  -  	\int_{ \bRfp } f \, \partial_{v_i}  f \partial_{v_i} 	\nu[f]  \, \md v \, \md I
 =: T_1 + T_2 - T_3 + T_4.
\end{multline*}
The term $T_1$ can be estimated using the Cauchy-Schwarz inequality and Theorem \ref{inc-reg},
\begin{equation*}
T_1 	\leq   \| Q_{b_2}^+(f,f) \|_{\dot{H}^{1}_{v}}  \ \| f \|_{\dot{H}^{1}_{v_i}} \leq   \Csvel   \, \|b_2\|_{L^{2}(\bSp)}\| f \|_{L^{2}_{(\gamma+2)^+}}^2   \| f \|_{\dot{H}^{1}_{v_i}}.
\end{equation*}
For the term $T_2$, Theorem \ref{inc-reg-b1} for $g=f$ implies
\begin{multline*}
	T_2 = 
	\int_{\bRfp } \partial_{v_i} f  \,\left( \Prim_i  + \R_i \right) \, \md  v \, \md I
	\leq 
	2			\int_{ \bRfp } | \partial_{v_i} f |    \, Q^+\left( \left| \nabla_{v}  f \right|,  f \right) \, \md  v \, \md I 
\\	+  \Ccmvel \, \| \tilde{b} \|_{{L^{1}(\bSp)}}  \|  \partial_{v_i} f  \|_{L^2}
	\| f \|^2_{L^{2}_{(\g+3/2)^+}}.
\end{multline*}
Then, $L^2$-theory \cite{MC-Alonso-Lp} and, in particular,   \eqref{Lp weak form} applied	  to the first term of $T_2$,  imply
\begin{equation*}
	T_2					  		\leq  \tilde{ \rho }\,  	\| \tilde{b} \|_{{L^{1}(\bSp)}}    \| \nabla_v f    \|^2_{L^2_{ {\g}/{2}}}  \|  	f    \|_{L^1_{\g}} 	+  \Ccmvel \, \| \tilde{b} \|_{{L^{1}(\bSp)}}  \|  \partial_{v_i} f  \|_{L^2}
	\| f \|^2_{L^{2}_{(\g+3/2)^+}}.
\end{equation*}
The estimate on $T_3$ follows from the lower bound on the collision frequency \eqref{coll freq lower est},
\begin{equation*}
	T_3 \geq \mA \,  \| \partial_{v_i} f    \|_{{L}^{2}_{\g/2}}^2.
\end{equation*}
Finally, term $T_4$ is  bounded by using the estimate on the derivative of the collision frequency \eqref{est derivative}, and the Cauchy-Schwartz inequality,
\begin{equation*}
	T_4  \leq\Clvel \,	\| f \|^2_{L^{2}_{(\g+3/2)^+}} \|  \partial_{v_i} f \|_{L^2}.
\end{equation*}
Gathering all estimates yields
\begin{multline*}
	\frac{1}{2} \,\frac{\md}{\md t}\| f \|_{\dot{H}^{1}_{v_i}}^2 \leq
	\Csvel \, \|b_2\|_{L^{2}}\| f \|_{L^{2}_{(\gamma+2)^+}}^2   \| f \|_{\dot{H}^{1}_{v_i}}
	\\
	+   \tilde{ \rho }\,  	\| \tilde{b} \|_{{L^{1}(\bSp)}}    \| \nabla_v f    \|^2_{L^2_{ {\g}/{2}}}  \|  	f    \|_{L^1_{\g}} 	+  \Ccmvel \, \| \tilde{b} \|_{{L^{1}(\bSp)}}  \|  \partial_{v_i} f  \|_{L^2}
	\| f \|^2_{L^{2}_{(\g+3/2)^+}}
	\\
+	\Clvel	\| f \|^2_{L^{2}_{(\g+3/2)^+}}   \| \partial_{v_i} f \|_{L^2}
	 -\mA \,  \| \partial_{v_i} f    \|_{{L}^{2}_{\g/2}}^2.
\end{multline*}
With Young's inequality and monotonicity of norms, the last term becomes
\begin{equation*}
	\frac{1}{2} \,\frac{\md}{\md t}\| f \|_{\dot{H}^{1}_{v_i}}^2 \leq
\frac{3}{2} \varepsilon \| f \|_{\dot{H}^{1}_{v_i}}^2
	+   \varepsilon \, \tilde{ \rho }\,   \| \nabla_v f    \|^2_{L^2_{ {\g}/{2}}}  \|  	f    \|_{L^1_{\g}} 	 
	- \mA \,  \| \partial_{v_i} f    \|_{{L}^{2}_{\g/2}}^2 +  \frac{\K}{3},
\end{equation*}
where the constant is given by
\begin{equation}\label{K}
\K = 
\sup_{t\geq0} \frac{3}{ \varepsilon} \| f \|_{L^{2}_{(\gamma+2)^+}}^4  \left(	\left(\Csvel \right)^2\|b_2\|_{L^{2}}^2
+   \varepsilon^2 \,    \left( \Ccmvel \right)^2   +   \left(\Clvel \right)^2   \right) 
\end{equation}
Summation over $i=1,2,3$, and noting that  $\| \nabla_{v} f    \|_{{L}^{2}_{\g/2}}^2  =  \sum_{i=1}^{3} \| \partial_{v_i} f    \|_{{L}^{2}_{\g/2}}^2$,
 imply
 \begin{equation*}
 	\begin{split}
	\frac{1}{2} \,\frac{\md}{\md t}\| f \|_{\dot{H}^{1}_{v}}^2  & \leq
	\frac{3}{2} \varepsilon \| f \|_{\dot{H}^{1}_{v}}^2
	+ 3\,  \varepsilon \, \tilde{ \rho }\,   \|  	f    \|_{L^1_{\g}} 	  \,  \| \nabla_v f    \|^2_{L^2_{ {\g}/{2}}} 
	- \mA\,  \| \nabla_{v} f    \|_{{L}^{2}_{\g/2}}^2 + 	\frac{\K}{2}
	\\
	& \leq  \varepsilon  \left( 	\frac{3}{2}    +    \tilde{ \rho }\,   \|  	f    \|_{L^1_{2}} 	   \right)  \| \nabla_v f    \|^2_{L^2_{ {\g}/{2}}} 
	- \mA \,  \| \nabla_{v} f    \|_{{L}^{2}_{\g/2}}^2 + 	\frac{\K}{2}.
	\end{split}
\end{equation*}
Thus, the choice of $\varepsilon$ as follows
\begin{equation}\label{epsilon}
	\varepsilon = \frac{1}{2 } \frac{\mA}{ \left( 	\frac{3}{2}    +    3 \, \tilde{ \rho }\,   \|  	f    \|_{L^1_{2}} 	   \right)},
\end{equation}
where $\tilde{\rho}$ is given in \eqref{Lp weak form}, 
together with the monotonicity  of norms, implies
\begin{equation*}
	\begin{split}
	 \frac{\md}{\md t}\| f \|_{\dot{H}^{1}_{v}}^2 \leq
		-  \mA\, \sum_{i=1}^{3} \| \partial_{v_i} f    \|_{{L}^{2}_{\g/2}}^2 +  \K \leq 	-  \mA \,  \| f \|_{\dot{H}^{1}_{v}}^2 + \K,
	\end{split}
\end{equation*}
where $\K$ is the one from \eqref{K} with the choice \eqref{epsilon}. This completes the proof of \eqref{regularity v} with the constant 
\begin{equation}\label{C reg vel}
\Crgvel =  \frac{\K}{\mA}.
\end{equation}
\end{proof}

\begin{proof}[Proof of Theorem \ref{Th regularity}, Part 2.] Starting from the Boltzmann equation \eqref{BE} and 
taking derivative 
with respect to $I$ implies
\begin{equation}\label{BE der I}
	\partial_t \partial_{I} f =   \partial_{I} Q^+(f,f)(v,I) - 	\nu[f] \partial_{I} f  - f \partial_{I} 	\nu[f].
\end{equation}
Multiplying the last equation \eqref{BE der I}  by $I^{2\delta} \partial_{I} f $ and integrating with respect to $(v,I)$ yields
\begin{multline*}
\frac{1}{2}	\,\frac{\md}{\md t}  \| I^{\delta} \partial_{I} f \|_{{L}^{2}}^2 = 	\int_{ \bRfp } I^{2\delta} (\partial_{I} f) \  \partial_{I} Q^+(f,f)  \, \md v \, \md I
\\ - 	\int_{ \bRfp }	\nu[f] \, |\partial_{I} f|^2  I^{2\delta} \, \md v \, \md I
-  	\int_{ \bRfp } I^{2\delta} f \, \partial_{I}  f \partial_{I} 	\nu[f]  \, \md v \, \md I
=: T_1 - T_2 - T_3.
\end{multline*}
Term $T_1$ becomes, after the Cauchy-Schwartz inequality and Part 2 of the Theorem \ref{inc-reg}, 
	\begin{equation*}
	\int_{\bRfp} \test \, \left(\partial_{I}		Q^+(f,f)(v,I) \right) \, I^{\delta}   \partial_{I} f \, \md v \, \md I 
	\leq  \Csen \,   \| b \|_{\Ls(\bSp)}\, \| f\|^2_{L^2_{(\g+2\delta+1/2)^+}}      \| I^{\delta}   \partial_{I} f \|_{L^{2}}.
\end{equation*}
The estimate from below of  the term $T_2$ follows from the lower bound lemma \ref{Prop: coll freq},
\begin{equation*}
		\int_{ \bRfp }	\nu[f] \, |\partial_{I} f|^2  I^{2\delta} \, \md v \, \md I \geq \mA	\int_{ \bRfp }	  |\partial_{I} f|^2  I^{2\delta} \,  \la v, I \ra^{\g}  \md v \, \md I = \mA \| I^{\delta} \partial_{I} f \|_{{L}^{2}_{\g/2}}^2.
\end{equation*}
For the term $T_3$, the estimate \eqref{est derivative energy} yields
\begin{equation*}
	T_3  \leq \Clen \,	\| f \|_{L^{2}_{(\g+1/2)^+}} 	\| f \|_{L^{2}_{(\g+2\delta+1/2)^+}} \|  I^{\delta} \partial_{I} f \|_{L^2}.
\end{equation*}
Gathering all estimates and exploiting the monotonicity of norms,  we obtain
\begin{equation*}
	\frac{1}{2}	\,\frac{\md}{\md t}  \| I^{\delta} \partial_{I} f \|_{{L}^{2}}^2 \leq  - \mA \| I^{\delta} \partial_{I} f \|_{{L}^{2}_{\g/2}}^2 
	+	\left( \Csen \| b \|_{\Ls(\bSp)}    + \Clen	
	\right) 	\| f \|_{L^{2}_{(\g+2\delta+1/2)^+}}^2   \|  I^{\delta} \partial_{I} f \|_{L^2}.
\end{equation*}
Young's inequality then implies 
\begin{equation*}
	\frac{1}{2}	\,\frac{\md}{\md t}  \| I^{\delta} \partial_{I} f \|_{{L}^{2}}^2 \leq  - \mA \| I^{\delta} \partial_{I} f \|_{{L}^{2}_{\g/2}}^2 
	\\ +  {\varepsilon} \|  I^{\delta} \partial_{I} f \|^2_{L^2} + \frac{K_2}{2},
\end{equation*}
with the constant
\begin{equation*}
	K_2 = \sup_{t\geq0} 	\| f \|_{L^{2}_{(\g+2\delta+1/2)^+}}^4   \frac{1}{\varepsilon} 	\left( \left( \Csen \right)^2  \| b \|^2_{\Ls(\bSp)}    + \left( \Clen \right)^2
	\right).
\end{equation*}
Monotonicity of norms and the choice $\varepsilon=\mA/2$ yield
\begin{equation*}
 \frac{\md}{\md t}  \| I^{\delta} \partial_{I} f \|_{{L}^{2}}^2 \leq  - \mA \| I^{\delta} \partial_{I} f \|_{{L}^{2}_{\g/2}}^2  + {K_2} \leq  - \mA \| I^{\delta} \partial_{I} f \|_{{L}^{2}}^2  + {K_2}.
\end{equation*}
This completes the proof of \eqref{regularity I} with the constant
\begin{equation}\label{C reg en}
\Crgen =	 \frac{K_2}{\mA}.
\end{equation}
\end{proof}

Arbitrary higher regularity propagation in the velocity variable can be achieved using the Leibniz formula for the gain collision operator, similar to the arguments in \cite{MV,AGT}.  
In contrast, arbitrary higher regularity in the internal energy variable $I$ is not possible, because the gain kernel deteriorates in the regions $r = 0$ and $r = 1$ after successive $I$-differentiations, due to the lack of a Leibniz formula for that variable.

\section{Proof of Theorem \ref{Th decomp} (Decomposition theorem)}\label{Sec proof decomp}
\begin{proof}[Proof of Therem \ref{Th decomp}]
We consider the Cauchy problem given by \eqref{BE} and \eqref{BE in}, where the Boltzmann equation is decomposed into its gain and loss terms,
\begin{equation}\label{BE initial}
	\begin{split}
	& \partial_t f(t,v,I) = 	Q^+(f(t,\cdot),f(t,\cdot))(v,I) - f(t,v,I) \	\nu[f(t,\cdot)](v,I), \quad t > 0,\\
	& f(0,v,I) = f_{0}(v,I).
	\end{split}
\end{equation}
For some $t_0 \geq 0$, recall the nonlinear Duhamel representation formula~\cite{MV}, valid for any $t \geq t_0$, where we use the notation $f(t_0, v, I) = f_{t_0}(v, I)$,
\begin{equation*}
	f(t,v,I)  = f_{t_0}(v,I) e^{-\int_{t_0}^t 	\nu[f(s,\cdot)](v,I) \md s}  + \int_{t_0}^t Q^+(f(s,\cdot),f(s,\cdot))(v,I) \, e^{-\int_s^t 	\nu[f(\tau,\cdot)](v,I) \, \md \tau} \, \md s.
\end{equation*}
Thus, for any time $t \geq t_0$, the solution $f$ can be decomposed into two non-negative terms, as shown in \eqref{f decomposed}, and given by
  \begin{equation*}
 	\begin{split}
\fR(t,v,I) & = f_{t_0}(v,I) e^{-\int_{t_0}^t 	\nu[f(s,\cdot)](v,I) \md s}, \\\
\fS(t,v,I) &= \int_{t_0}^t Q^+(f(s,\cdot),f(s,\cdot))(v,I) \, e^{-\int_s^t 	\nu[f(\tau,\cdot)](v,I) \, \md \tau} \, \md s,
\end{split}
 \end{equation*}
which we refer to as the\emph{rough}  and  \emph{smooth}  parts, respectively. As we will show later, the rough part $\fR(t,v,I)$ retains the regularity properties of the  data $f_{t_0}(v,I)$ at time $t_0 \geq 0$,  without gaining smoothness, but its contribution decays exponentially in time in any norm. In contrast, the smooth part $\fS(t,v,I)$ exhibits improved regularity, with smooth first derivatives under appropriate assumptions on the initial data, as described in parts~(a) and~(b) of Theorem \ref{Th decomp}.

First, notice that the lower bound on the collision frequency \eqref{coll freq lower est} implies 
\begin{equation}\label{coll freq lower est 2} 
	\nu[f](v,I) \,   \geq \mA \, \la v, I \ra^{\g} > \mA > 0, \ \text{for any } (v,I) \in \bRfp, \text{with} \ \g>0.
\end{equation} 
Thus, the following pointwise bound on the rough part $\fR$ holds,
\begin{equation*}
\fR(t,v,I)  \leq e^{-\mA (t-t_0)} f_{t_0}(v,I),
\end{equation*}
which proves \eqref{rough est}. 

To prove the smoothness properties of $\fS$, we first consider  the velocity variable. 
Taking the derivative of $\fS$  with respect to $v_i$ yields two terms
\begin{equation}\label{fS v}
\partial_{v_i}	\fS(t,v,I) = \mathcal{I}_1 + \mathcal{I}_2, \quad t\geq t_0,
\end{equation}
where 
\begin{equation*}
	\begin{split}
		 \mathcal{I}_1  &  =  \int_{t_0}^t  Q^+(f(s,\cdot),f(s,\cdot))(v,I) \, \left( \partial_{v_i} e^{-\int_s^t 	\nu[f(\tau,\cdot)](v,I) \, \md \tau} \right) \, \md s, \\ 
\mathcal{I}_2 & =  \int_{t_0}^t \left(  \partial_{v_i} Q^+(f(s,\cdot),f(s,\cdot))(v,I) \right) \, e^{-\int_s^t 	\nu[f(\tau,\cdot)](v,I) \, \md \tau} \, \md s.
	\end{split}
\end{equation*}
For the term $\mathcal{I}_1  $,  the estimate on the velocity derivative of the collision frequency \eqref{est derivative} is exploited,  
\begin{equation*}
	\begin{split}
		\mathcal{I}_1  
		& \leq  \int_{t_0}^t  Q^+(f(s,\cdot),f(s,\cdot))(v,I) \,  e^{-\int_s^t 	\nu[f(\tau,\cdot)](v,I) \, \md \tau}    \int_s^t  \left| \partial_{v_i} 	\nu[f(\tau,\cdot)](v,I)  \right| \, \md \tau  \, \md s \\
		& \leq  \Clvel \, \cb_{(\g -1)/2} \sup_{\tau> t_0}	\| f(\tau,\cdot) \|_{L^{2}_{(\g+3/2)^+}}   \\ & \qquad \qquad \qquad  \qquad \qquad \times \int_{t_0}^t  Q^+(f(s,\cdot),f(s,\cdot))(v,I) \,  e^{- \mA (t-s)  \la v, I \ra^{\g}  }    (t-s) \la v, I \ra^{(\g+3/2)^+}    \md s.
	\end{split}
\end{equation*}
Since $ x \, e^{- \mA x } \leq e^{- \mA x/2 }$ for all $x\geq 0$, the last term can be simplified to
\begin{equation}\label{pomocna 2}
	\begin{split}
		\mathcal{I}_1  
	\leq \Clvel \sup_{\tau>t_0}	\| f(\tau,\cdot) \|_{L^{2}_{(\g+3/2)^+}}    \int_{t_0}^t  Q^+(f(s,\cdot),f(s,\cdot))(v,I) \la v, I \ra^{(3/2)^+}   \,  e^{- \frac{\mA}{2} (t-s)     }     \md s.
	\end{split}
\end{equation}
To attach the weight next to the collision operator to the input functions, we observe that the energy conservation identity \eqref{coll CL} implies
\begin{equation*}
	\langle v, I \rangle \leq \langle v', I' \rangle \, \langle v'_*, I'_* \rangle.
\end{equation*}
This, in turn, allows us to write, defining $\tilde{f}(s,v,I) := f(s,v,I) \langle v, I \rangle^{(3/2)^+}$, 
\begin{equation*}
	Q^+(f(s,\cdot), f(s,\cdot))(v,I) \, \langle v, I \rangle^{(3/2)^+} \leq Q^+(\tilde{f}(s,\cdot), \tilde{f}(s,\cdot))(v,I).
\end{equation*}
Moreover,  applying the $L^2$-estimate \eqref{Lp Q+}, 
\begin{equation*}
	\begin{split}
		\| Q^+(\tilde{f}(s,\cdot),\tilde{f}(s,\cdot)) \|_{L^2} & \leq 	    \rho_3	\| b \|_{\Ls}   \| \tilde{f}(s,\cdot)    \|_{L^2_{ \g}}  \|  	\tilde{f}(s,\cdot)    \|_{L^1_{\g}} \\ & =   \rho_3	\| b \|_{\Ls}   \|  {f}(s,\cdot)    \|_{L^2_{(\g+3/2)^+}}  \|  	{f}(s,\cdot)    \|_{L^1_{(\g+3/2)^+}}.
	\end{split}
\end{equation*}
Substituting this into \eqref{pomocna 2} leads to the final estimate
\begin{equation}\label{I1 L2}
	\| \mathcal{I}_1 \|_{L^2} \leq  \frac{2 \, \Clvel}{\mA}   \,    \rho_3 \,  	\| b \|_{\Ls}\, \sup_{s>t_0}	\| f(s,\cdot) \|_{L^{1}_{(\g+3/2)^+}} \, \sup_{s>t_0}	\| f(s,\cdot) \|^2_{L^{2}_{(\g+3/2)^+}}.
\end{equation}
For the term $\mathcal{I}_2 $, 
the estimate \eqref{coll freq lower est 2} yields
\begin{equation*}
\|	\mathcal{I}_2  \|_{L^2} \leq  \int_{t_0}^t  e^{- \mA (t-s)} \|  \partial_{v_i} Q^+(f(s,\cdot),f(s,\cdot)) \|_{L^2}  \, \md s.
\end{equation*}
 In addition, Theorem~\ref{inc-reg}, and specifically the estimate \eqref{eq Th smooth vel}, provides the bound
\begin{equation}\label{I2 L2}
	\begin{split}
	\|	\mathcal{I}_2  \|_{L^2} & \leq  \Csvel  \, \|b\|_{L^{2}} \int_{t_0}^t  e^{- \mA (t-s)}  \| f(s,\cdot) \|^2_{L^{2}_{(\gamma+2)^+}} \, \md s
	\leq \frac{ \Csvel}{\mA} \, \|b\|_{L^{2}} \, \sup_{s > t_0 } \| f(s,\cdot) \|^2_{L^{2}_{(\gamma+2)^+}}.
	\end{split}
\end{equation}
In conclusion,  combining the estimates \eqref{I1 L2} and \eqref{I2 L2}, and using the monotonicity of norms, \eqref{fS v} implies the regularity estimate of $\fS$ with respect to the velocity variable,  for any $t \geq t_0$,
\begin{multline}\label{v reg of fS}
	\| 	\fS(t,\cdot) \|_{\dot{H}^{1}_{ v}} \leq \frac{\sqrt{3}}{\mA} 
 \left(  2 \, \Clvel  \,    \rho_3 \,  	\| b \|_{\Ls}\, \sup_{s>t_0}	\| f(s,\cdot) \|_{L^{1}_{\g+2}} \, +   \Csvel \|b\|_{L^{2}} \right) \sup_{s >t_0 } \| f(s,\cdot) \|^2_{L^{2}_{(\gamma+2)^+}}.
\end{multline}
It remains to prove the finiteness of the term on the right-hand side; this will be done after first deriving an estimate on the regularity with respect to the $I$ variable.

Indeed, taking the derivative of $\fS$  with respect to $I$ and multiplying with $I^\delta$,  where $\delta$ is from Part 2 of Theorem \ref{inc-reg}, 
 \begin{equation}\label{fS I}
	I^\delta \partial_{I}	\fS(t,v,I)  = \mathcal{J}_1 + \mathcal{J}_2, \quad t\geq t_0,
\end{equation}
where 
\begin{equation*}
	\begin{split}
		\mathcal{J}_1  &  =  I^\delta \int_{t_0}^t  Q^+(f(s,\cdot),f(s,\cdot))(v,I) \, \left( \partial_{I} e^{-\int_s^t 	\nu[f(\tau,\cdot)](v,I) \, \md \tau} \right) \, \md s, \\ 
		\mathcal{J}_2 & = I^\delta  \int_{t_0}^t \left(  \partial_{I} Q^+(f(s,\cdot),f(s,\cdot))(v,I) \right) \, e^{-\int_s^t 	\nu[f(\tau,\cdot)](v,I) \, \md \tau} \, \md s.
	\end{split}
\end{equation*}
Similar reasoning as for $\mathcal{I}_1$, implies, by means of \eqref{est derivative energy}, 
\begin{equation*}
	\begin{split}
		\mathcal{J}_1  
		& \leq  \Clen \sup_{\tau>t_0}	\| f(\tau,\cdot) \|_{L^{2}_{(\g+1/2)^+}}    \int_{t_0}^t  Q^+(f(s,\cdot),f(s,\cdot))(v,I) \la v, I \ra^{(2\delta + 1/2)^+}   \,  e^{- \frac{\mA}{2} (t-s)     }     \, \md s. 
	\end{split}
\end{equation*}
Then, since
 \begin{equation*}
 	\begin{split}
 		\|  Q^+(f(s,\cdot),f(s,\cdot)) \la \cdot \ra^{(2\delta + 1/2)^+}   \|_{L^2} & \leq 	    \rho_3	\| b \|_{\Ls}   \|  {f}(s,\cdot)    \|_{L^2_{(\g+2\delta+1/2)^+}}  \|  	{f}(s,\cdot)    \|_{L^1_{(\g+2\delta+1/2)^+}},
 	\end{split}
 \end{equation*}
the monotonicity of norms implies the final estimate
\begin{equation}\label{J1 L2}
	\|	\mathcal{J}_1  \|_{L^2} 
		\leq   \frac{2\, \Clen}{\mA}  \, \rho_3 \,  	\| b \|_{\Ls}    \sup_{s>t_0}  \|  	{f}(s,\cdot)    \|_{L^1_{(\g+2\delta+1/2)^+}}  \sup_{s>t_0}  \|  {f}(s,\cdot)    \|^2_{L^2_{(\g+2\delta+1/2)^+}}.
\end{equation}
For the term $\mathcal{J}_2$, the lower bound \eqref{coll freq lower est 2} together with the estimate \eqref{eq Th smooth int en} from Theorem~\ref{inc-reg} imply
\begin{equation}\label{J2 L2}
	\begin{split}
	\|	\mathcal{J}_2  \|_{L^2}  & \leq  \int_{t_0}^t  e^{- \mA (t-s)} \| I^\delta \partial_{I} Q^+(f(s,\cdot),f(s,\cdot)) \|_{L^2}  \, \md s 	 \leq \frac{ \Csen}{\mA} \| b \|_{\Ls} \, \sup_{s>t_0} \| f(s,\cdot)\|^2_{L^2_{(\g+2\delta+1/2)^+}}.
	 \end{split}
\end{equation}
Thus, combining \eqref{J1 L2} and \eqref{J2 L2} into \eqref{fS I}, with monotonicity of norms, yields the regularity estimate for $\fS$ with respect to the $I$ variable, valid for any $t \geq t_0$,
 \begin{equation}\label{I reg of fS}
	\| 	I^\delta \partial_{I}	\fS \|_{L^2} \leq \frac{	\| b \|_{\Ls} }{\mA}  \left(    2\, \Clen  \, \rho_3 \,      \sup_{s>t_0}  \|  	{f}(s,\cdot)    \|_{L^1_{\g+2}} +   \Csen \right) \sup_{s> t_0} \| f(s,\cdot)\|^2_{L^2_{(\g+2)^+}},
\end{equation}
where the finiteness of the right-hand side will be proven below. 

To conclude the proof of \eqref{v reg of fS} and \eqref{I reg of fS}, we need to show the finiteness of the right-hand sides, depending on the integrability properties of the initial data. 

Part (a) assumes $f_0(v,I) \in L^1_2 \cap L^2$. Then, for any $t\geq t_0 >0$,  generation of $L^1$-moments \eqref{poly gen} implies
\begin{equation}\label{poly gen 2}
	\| f(s,\cdot) \|_{L^1_k} \leq C^{\text{gen}} \left(1  +  s^{-\frac{k-2}{\g}} \right) \quad \Rightarrow \quad  \sup_{s>t_0}  \|  	{f}(s,\cdot)    \|_{L^1_{\g+2}} \leq  C^{\text{gen}}  \left(1  +  t_0^{-1} \right) < \infty \ \text{for} \ t_0>0.
\end{equation}
Moreover, generation of $L^p$ tails \eqref{Lp generation t0} for $p=2$ implies
\begin{multline*}
	 \| f(s,\cdot)\|_{L^2_{k}} \leq C^{\text{gen}}_{2; t_0}  (1+ (s-\tilde{t})^{-\frac{k}{\g}}), \quad s>\tilde{t} := \frac{t_0}{2} \\ \Rightarrow \quad \sup_{s> t_0} \| f(s,\cdot)\|^2_{L^2_{(\g+2)^+}} \leq C^{\text{gen}}_{2; t_0}  \Big(1+ \Big(\frac{t_0}{2}\Big)^{-\frac{(\g+2)^+}{\g}}\Big) < \infty \ \text{for} \ t_0>0,
\end{multline*}
which completes the proof of \eqref{fS gen}, with the constants
\begin{equation*}
	\begin{split}
\CveltS & = \frac{\sqrt{3}}{\mA} \left( 2  \, \Clvel  \,  \rho_3 \,  	\| b \|_{\Ls}\,C^{\text{gen}}  \left(1  +  t_0^{-1} \right) +   \Csvel \, \|b\|_{L^{2}} \right)\,C^{\text{gen}}_{2; t_0}  \Big(1+ \Big(\frac{t_0}{2}\Big)^{-\frac{(\g+2)^+}{\g}}\Big), \\
	\CenttS & = \frac{	\| b \|_{\Ls} }{\mA}  \left(    2 \, \Clen \, \rho_3 \,      C^{\text{gen}}  \left(1  +  t_0^{-1} \right)  +   \Csen \right) C^{\text{gen}}_{2; t_0}  \Big(1+ \Big(\frac{t_0}{2}\Big)^{-\frac{(\g+2)^+}{\g}}\Big).
	\end{split}
\end{equation*}

Part~(b) assumes $f_0(v,I) \in L^1_{(2\gamma+3)^+} \cap L^2_{(\gamma+2)^+}$, which allows   to conclude, by the propagation of moments \eqref{poly prop} and the $L^2$ propagation estimate \eqref{Lp propagation}, that for any time $t \geq 0$, the right-hand sides of \eqref{v reg of fS} and \eqref{I reg of fS} are finite, 
thereby completing the proof of \eqref{fS prop}, with the constants
\begin{equation*}
	\begin{split}
	\CvelzS & = \frac{\sqrt{3}}{\mA} \left( 2  \, \Clvel  \,  \rho_3 \,  	\| b \|_{\Ls}\, \max \left\{  e \,	\| f_0 \|_{L^1_{\g+2}}, \Cpr \right\} +   \Csvel \, \|b\|_{L^{2}} \right)\, \max\left\{  \| f_0    \|_{L^2_{(\gamma+2)^+}}, \Cpr_2 \right\}, \\
	\CenzS & = \frac{	\| b \|_{\Ls} }{\mA}  \left(    2 \, \Clen \, \rho_3 \,      \max \left\{  e \,	\| f_0 \|_{L^1_{\g+2}}, \Cpr \right\}  +   \Csen \right)  \max\left\{  \| f_0    \|_{L^2_{(\gamma+2)^+}}, \Cpr_2 \right\}.
	\end{split}
\end{equation*}

\end{proof}

\appendix

\section{Auxiliary results}

The  Cauchy-Schwartz inequality implies the following estimate, for any suitable $g$,
\begin{equation}\label{constant Ca estimate}
	\int_{\bRfp} g(v_*,I_*)   \E^{a} \md v_* \, \md I_* \leq \cb_a \,  \| g\|_{L^2_{\left(2 a + {5}/{2}\right)^+}} \la v, I \ra^{\left(2 a + {5}/{2}\right)^+}, \quad \text{for} \ a>-\frac{5}{4},
\end{equation}
where the following  constant is introduced,
\begin{equation}\label{constant Ca}
	\mathcal{C}_a =	 \sup_{(v,I)}  \left\| \    \E^{a} \la v, I \ra^{-s}  \right\|_{L^2_{-s}(\md v_* \md I_*)}, \quad \text{for} \ a>-\frac{5}{4}, \ \text{and} \ s> 2 a + \frac{5}{2}.
\end{equation}

\begin{lemma}\label{Lemma F est} 
	For any $x \in \mathbb{R}$ and $\alpha\geq 0$,
	\begin{equation}\label{F est lemma}
		\left| \int_{r\in[0,1]} e^{-i  \, x \, r} \,  r^\alpha(1-r)^\alpha \md r  \right| \leq  2 \min \left\{ 1, \frac{1}{|x|} \right\}.
	\end{equation}
\end{lemma}
\begin{proof}
	First, for $\alpha=0$, one can explicitly compute the integral and estimate
	\begin{equation}\label{r alpha=0}
		\left| \int_{r\in[0,1]} e^{-i \, x  \, r}    \md r  \right| \leq  \frac{2}{|x|}.
	\end{equation}
	When $\alpha>0$, on one hand, note
	\begin{equation}\label{r alpha>0 1}
		\left| \int_{r\in[0,1]} e^{-i \, x \, r }  r^\alpha(1-r)^\alpha \md r \right| \leq  \int_{r\in[0,1]}  r^\alpha(1-r)^\alpha \md r,
	\end{equation}
	and, on the other hand,  integration by parts implies
	\begin{equation}\label{r alpha>0 2}
		\left| \int_{r\in[0,1]} e^{-i \,  x\,  r }  r^\alpha(1-r)^\alpha \md r  \right| \leq  \frac{1}{|x|  }  \int_{r\in[0,1]} \alpha \left( r^{\alpha-1}(1-r)^\alpha + r^\alpha (1-r)^{\alpha-1} \right) \md r.
	\end{equation}
	Since for $\alpha>0$
	\begin{equation*}
		\int_{r\in[0,1]}  r^\alpha(1-r)^\alpha \md r \leq \int_{r\in[0,1]} \alpha \left( r^{\alpha-1}(1-r)^\alpha + r^\alpha (1-r)^{\alpha-1} \right) \md r \leq 2,
	\end{equation*}
	we   gather the estimates  \eqref{r alpha=0}, \eqref{r alpha>0 1}, \eqref{r alpha>0 2}, and conclude \eqref{F est lemma} for any $\alpha\geq 0$.
\end{proof}

\section{Toolbox from $L^p$ theory}

\subsection{Estimates on the gain term}
For any suitable function $\psi(r,R)$, we define the following averaging operator
\begin{equation}\label{S operator gen}
	\mS^\psi(\chi)(v,I, v_*, I_*)  =  \int_{  \domparam }  b(\hat{u}\cdot\sigma) \,  \psi(r,R)  \, \chi(v',I')  \, \dpo \, \md \sigma\, \md R \,\md r.
\end{equation}
The estimate on its $L^2$-norm can be obtained slightly  modifying the one from \cite{MC-Alonso-Lp}. Namely, defining a constant, assumed finite for a proper choice of $\psi$,
\begin{equation}\label{rho gen}
	 \rho^\psi  = 2^{7/4}   \int_{[0,1]^2}  \frac{1}{\sqrt{r(1-R)}}\, \psi(r,R) \, \dub \, \md r \, \md R,
\end{equation}
the following estimate holds, for $b \in L^1$,
\begin{equation}\label{S estimate L1 gen}
	\sup_{(v_*, I_*)}	\left\| 	\mS^\psi(\chi) \right\|_{L^{2}(\md v\, \md I)}  \leq   \rho^\psi \,   \| b \|_{\Ls(\bSp)}   \| \chi\|_{L^{2}}.
\end{equation}
Now, choose  the collision kernel \eqref{coll kernel assumpt}. We state two results from \cite{MC-Alonso-Lp}. 

First, Proposition 6.1 for $p=q=2$, implies, for $\alpha \geq 0$,
 	\begin{equation}\label{Lp weak form}
	\int_{\bRfp} Q^+(f, g)(v, I) \, \chi(v, I) \, \md v\, \md I  
	\\
	\leq \tilde{\rho} \,	\| b \|_{\Ls}   \| f    \|_{L^p_{ {\g}/{2}}}  \|  	g    \|_{L^1_{\g}} \| \chi  \|_{L^{2}_{{\g}/{2}}}, \quad \tilde{\rho} := 2^{\frac{3\g }{4}} { \rho_1 },
\end{equation}
where $\rho_1$ is $ \rho^\psi $ for  $\psi = \left( \frac{1}{(1-r)(1-R)} + \frac{1}{\sqrt{R}}\right)$.

Second, since 
	\begin{equation}\label{Lp weak form L2}
	\int_{\bRfp} Q^+(f, g)(v, I) \, \chi(v, I) \, \md v\, \md I  
	\\
	\leq    \rho_3	\| b \|_{\Ls}   \| f    \|_{L^2_{ \g}}  \|  	g    \|_{L^1_{\g}} \| \chi  \|_{L^{2}},
\end{equation}
where $\rho_3$ is $\rho^\psi$ for $\psi=1$, finite for  $\alpha>-1/2$, by duality it holds
	\begin{equation}\label{Lp Q+}
\| Q^+(f, g) \|_{L^2}
	\leq   \rho_3	\| b \|_{\Ls}   \| f    \|_{L^2_{ \g}}  \|  	g    \|_{L^1_{\g}}.
\end{equation}

.

\section*{Acknowledgments}
R. A. thanks TAMUQ internal funding research grant 470242-25650.  M. \v{C}. acknowledges
the support from the European Union’s Horizon Europe research and innovation programme, under
the Marie Sk{\l}odowska-Curie grant agreement No. 101194202, project  \emph{Boltzmann models for polyatomic gases and mixtures: analysis, macroscopic limits and entropy methods} (BAME).
	
\bibliography{polyatomic}

\providecommand{\noopsort}[1]{}\providecommand{\singleletter}[1]{#1}%
\begin{thebibliography}{10}

\bibitem{non-cut}
R.~Alexandre, Y.~Morimoto, S.~Ukai, C.-J. Xu, and T.~Yang.
\newblock Smoothing effect of weak solutions for the spatially homogeneous
  {B}oltzmann equation without angular cutoff.
\newblock {\em Kyoto J. Math.}, 52(3):433--463, 2012.

\bibitem{Alonso-cooling}
R.~Alonso, V.~Bagland, Y.~Cheng, and B.~Lods.
\newblock One-dimensional dissipative {B}oltzmann equation: measure solutions,
  cooling rate, and self-similar profile.
\newblock {\em SIAM J. Math. Anal.}, 50(1):1278--1321, 2018.

\bibitem{AGT}
R.~Alonso, I.~M. Gamba, and M.~Tasković.
\newblock Exponentially-tailed regularity and time asymptotic for the
  homogeneous {B}oltzmann equation.
\newblock {\em ArXiv:1711.06596}, 2024.

\bibitem{MC-Alonso-Pesaro}
R.~Alonso and M.~\v{C}oli\'{c}.
\newblock Boltzmann framework for polyatomic gases: review on well-posedness,
  higher integrability and physical relevance.
\newblock {\em ArXiv:2502.20308}, 2025.

\bibitem{MC-Alonso-frozen}
R.~Alonso and M.~\v{C}oli\'{c}.
\newblock Moment estimates for polyatomic {B}oltzmann equation with frozen
  collisions.
\newblock {\em ArXiv:2502.08237}, 2025.

\bibitem{AGT-S}
R.~J. Alonso, I.~M. Gamba, and S.~H. Tharkabhushanam.
\newblock Convergence and error estimates for the {L}agrangian-based
  conservative spectral method for {B}oltzmann equations.
\newblock {\em SIAM J. Numer. Anal.}, 56(6):3534--3579, 2018.

\bibitem{MC-Alonso-Lp}
R.~J. Alonso and M.~\v{C}oli\'c.
\newblock Integrability propagation for a {B}oltzmann system describing
  polyatomic gas mixtures.
\newblock {\em SIAM J. Math. Anal.}, 56(1):1459--1494, 2024.

\bibitem{MC-Alonso-Gamba}
R.~J. Alonso, M.~\v{C}oli\'c, and I.~M. Gamba.
\newblock {T}he {C}auchy {P}roblem for {B}oltzmann {B}i-linear {S}ystems: {T}he
  {M}ixing of {M}onatomic and {P}olyatomic {G}ases.
\newblock {\em J. Stat. Phys.}, 191(9), 2024.

\bibitem{Bern}
N.~Bernhoff.
\newblock Linearized {B}oltzmann collision operator: {II}. {P}olyatomic
  molecules modeled by a continuous internal energy variable.
\newblock {\em Kinet. Relat. Models}, 16(6):828--849, 2023.

\bibitem{Bern-nous}
N.~Bernhoff, L.~Boudin, M.~\v{C}oli\'{c}, and B.~Grec.
\newblock Compactness of linearized {B}oltzmann operators for polyatomic gases.
\newblock {\em arXiv:2407.11452}, 2025.

\bibitem{Bob-Exact-solutions-Fourier}
A.~V. Bobylev.
\newblock Exact solutions of the nonlinear {B}oltzmann equation and the theory
  of relaxation of a {M}axwellian gas.
\newblock {\em Theor. Math. Phys.}, 60(2):820--841, 1984.

\bibitem{Bor-Larsen}
C.~Borgnakke and P.~S. Larsen.
\newblock Statistical collision model for {M}onte {C}arlo simulation of
  polyatomic gas mixture.
\newblock {\em J. Comput. Phys.}, 18(4):405--420, 1975.

\bibitem{BD}
F.~Bouchut and L.~Desvillettes.
\newblock A proof of the smoothing properties of the positive part of
  {B}oltzmann's kernel.
\newblock {\em Rev. Mat. Iberoamericana}, 14(1):47--61, 1998.

\bibitem{LD-Bourgat}
J.-F. Bourgat, L.~Desvillettes, P.~Le~Tallec, and B.~Perthame.
\newblock Microreversible collisions for polyatomic gases and {B}oltzmann's
  theorem.
\newblock {\em European J. Mech. B Fluids}, 13(2):237--254, 1994.

\bibitem{Brull-Comp-2}
S.~Brull, M.~Shahine, and P.~Thieullen.
\newblock Fredholm property of the linearized {B}oltzmann operator for a
  polyatomic single gas model.
\newblock {\em Kinet. Relat. Models}, 2023.

\bibitem{Brull}
S.~Brull, M.~Shahine, and P.~Thieullen.
\newblock Fredholm property of the linearized {B}oltzmann operator for a
  polyatomic single gas model.
\newblock {\em Kinet. Relat. Models}, 17(2):234--252, 2024.

\bibitem{LD-Toulouse}
L.~Desvillettes.
\newblock Sur un mod\`{e}le de type {B}orgnakke–{L}arsen conduisant \`{a} des
  lois d’energie non-lin\'{e}aires en temp\'{e}rature pour les gaz parfaits
  polyatomiques.
\newblock {\em Ann. Fac. Sci. Toulouse Math.}, 6(0):257--262, 1997.

\bibitem{LD-Parma}
L.~Desvillettes.
\newblock Boltzmann’s kernel and the spatially homogeneous {B}oltzmann
  equation.
\newblock {\em Riv. Mat. Univ. Parma}, 4*(6):1--22, 2001.

\bibitem{DesMonSalv}
L.~Desvillettes, R.~Monaco, and F.~Salvarani.
\newblock A kinetic model allowing to obtain the energy law of polytropic gases
  in the presence of chemical reactions.
\newblock {\em Eur. J. Mech. B Fluids}, 24(2):219--236, 2005.

\bibitem{MPC-Dj-T-O}
V.~Djordji\'{c}, G.~Oblapenko, M.~Pavi\'{c}-\v{C}oli\'{c}, and M.~Torrilhon.
\newblock {B}oltzmann collision operator for polyatomic gases in agreement with
  experimental data and {DSMC} method.
\newblock {\em Contin. Mech. Thermodyn.}, 35:103--119, 2023.

\bibitem{MPC-Dj-S}
V.~Djordji\'{c}, M.~Pavi\'{c}-\v{C}oli\'{c}, and N.~Spasojevi\'{c}.
\newblock Polytropic gas modelling at kinetic and macroscopic levels.
\newblock {\em Kinet. Relat. Models}, 14(3):483--522, 2021.

\bibitem{MPC-Dj-T}
V.~Djordji\'{c}, M.~Pavi\'{c}-\v{C}oli\'{c}, and M.~Torrilhon.
\newblock Consistent, explicit and accessible {B}oltzmann collision operator
  for polyatomic gases.
\newblock {\em Phys. Rev. E}, 104:025309, 2021.

\bibitem{Duan-Li}
R.~Duan and Z.~Li.
\newblock Global bounded solutions to the {B}oltzmann equation for a polyatomic
  gas.
\newblock {\em Internat. J. Math.}, 34(7):Paper No. 2350036, 43, 2023.

\bibitem{MPC-IG-poly}
I.~M. Gamba and M.~Pavi\'{c}-\v{C}oli\'{c}.
\newblock On the {C}auchy problem for {B}oltzmann equation modeling a
  polyatomic gas.
\newblock {\em J. Math. Phys.}, 64:013303, 2023.

\bibitem{Gio}
V.~Giovangigli.
\newblock {\em Multicomponent flow modeling}.
\newblock Modeling and Simulation in Science, Engineering and Technology.
  Birkh\"{a}user Boston, Inc., Boston, MA, 1999.

\bibitem{Ko-Son}
G.~Ko and S.-j. Son.
\newblock Global stability of the {B}oltzmann equation for a polyatomic gas
  with initial data allowing large oscillations.
\newblock {\em J. Differential Equations}, 425:506--552, 2025.

\bibitem{L}
P.-L. Lions.
\newblock Compactness in {B}oltzmann's equation via {F}ourier integral
  operators and applications. {III}.
\newblock {\em J. Math. Kyoto Univ.}, 34(3):539--584, 1994.

\bibitem{LU}
X.~Lu.
\newblock A direct method for the regularity of the gain term in the
  {B}oltzmann equation.
\newblock {\em J. Math. Anal. Appl.}, 228(2):409--435, 1998.

\bibitem{Martin-ODE}
R.~H. Martin, Jr.
\newblock {\em Nonlinear operators and differential equations in {B}anach
  spaces}.
\newblock Pure and Applied Mathematics. Wiley-Interscience [John Wiley \&
  Sons], New York-London-Sydney, 1976.

\bibitem{MV}
C.~Mouhot and C.~Villani.
\newblock Regularity theory for the spatially homogeneous {B}oltzmann equation
  with cut-off.
\newblock {\em Arch. Ration. Mech. Anal.}, 173(2):169--212, 2004.

\bibitem{Kusto-book}
E.~Nagnibeda and E.~Kustova.
\newblock {\em Non-equilibrium reacting gas flows}.
\newblock Heat and Mass Transfer. Springer-Verlag, Berlin, 2009.

\bibitem{MPC-SS-non-poly}
M.~Pavi\'c-\v Coli\'c and S.~Simi\'c.
\newblock Kinetic description of polyatomic gases with temperature-dependent
  specific heats.
\newblock {\em Phys. Rev. Fluids}, 7:083401, 2022.

\bibitem{Rugg-book-2}
T.~Ruggeri and M.~Sugiyama.
\newblock {\em Classical and relativistic rational extended thermodynamics of
  gases}.
\newblock Springer, 2021.

\bibitem{Str}
H.~Struchtrup.
\newblock {\em Macroscopic Transport Equations for Rarefied Gas Flows}.
\newblock Interaction of Mechanics and Mathematics. Springer, Berlin, 2005.

\bibitem{torrilhon2016modeling}
M.~Torrilhon.
\newblock Modeling nonequilibrium gas flow based on moment equations.
\newblock {\em Annu. Rev. Fluid Mech.}, 48(1):429--458, 2016.

\bibitem{V}
C.~Villani.
\newblock A {R}eview of {M}athematical {T}opics in {C}ollisional {K}inetic
  {T}heory.
\newblock volume~1 of {\em Handbook of Mathematical Fluid Dynamics}, pages
  71--305. North-Holland, 2002.

\bibitem{W}
B.~Wennberg.
\newblock Regularity in the {B}oltzmann equation and the {R}adon transform.
\newblock {\em Commun. Partial. Differ. Equ.}, 19(11-12):2057--2074, 1994.

\end{thebibliography}

\end{document}